\documentclass[11pt, reqno]{amsart}

\usepackage{esint}
\usepackage{amstext, mathtools}
\usepackage{amsthm}
\usepackage{amsmath}
\usepackage{amssymb}
\usepackage{latexsym}
\usepackage{amsfonts}
\usepackage{graphicx}
\usepackage{color}

\usepackage[mathscr]{euscript}
\usepackage[pagebackref,hypertexnames=false, colorlinks, citecolor=red, linkcolor=red]{hyperref}

{\end{list}}

\newcommand{\D}{\mathbb{D}}

\newcommand{\T}{\mathbb{T}}
\newcommand{\R}{\mathbb{R}}

\newcommand{\C}{\mathbb{C}}

\newcommand{\abf}[1]{\vartheta^{#1}}

\newcommand{\tay}{\boldsymbol{\sigma}}

\newcommand{\McC}{\raise.5ex\hbox{c}}

\newcounter{vremennyj}

\numberwithin{equation}{section}

\newtheorem{thm}{Theorem}[section]
\newtheorem{lm}[thm]{Lemma}
\newtheorem{cor}[thm]{Corollary}

\newtheorem{prop}[thm]{Proposition}
\newtheorem*{prop*}{Proposition}

\theoremstyle{remark}
\newtheorem{rem}[thm]{Remark}
\newtheorem*{rem*}{Remark}

\newtheorem{example}[thm]{Example}

\begin{document}

\title[]{Notes}

\author[Arora]{Palak Arora}
\email{palak.official94@gmail.com}

\author[Bickel]{Kelly Bickel$^\dagger$}
\address{Department of Mathematics, Bucknell University, 360 Olin Science Building, Lewisburg, PA 17837, USA.}
\email{kelly.bickel@bucknell.edu}
\thanks{$\dagger$ Research supported in part by National Science Foundation DMS grant \#2000088.}

\author[Liaw]{Constanze Liaw$^\ddagger$}
 \thanks{$^\ddagger$ Research supported in part by National Science Foundation DMS grant \#2452894.}
\address{Department of Mathematical Sciences, University of Delaware, 517A Ewing Hall, Newark, DE 19716, USA.}
\email{liaw@udel.edu}

\author[Sola]{Alan Sola}
\address{Department of Mathematics, Stockholm University, 106 91 Stockholm, Sweden}
\email{sola@math.su.se}


\title[Clark Unitaries and Taylor Joint Spectra]{Pairs of Clark Unitary Operators on the Bidisk and their Taylor Joint Spectra}
\date{\today}
\begin{abstract}
We develop a Clark theory for commuting compressed shift operators on model spaces $K_{\phi}$ associated with inner functions $\phi$ on the bidisk, which exhibits both similarities and marked differences compared to the classical one-variable version. We first identify the adjoint of the embedding operator $J_{\alpha} \colon K_{\phi}\to L^2(\sigma_{\alpha})$ as a weighted Cauchy transform of the Clark measure $\sigma_{\alpha}$. Under natural assumptions, which generically include the case when $\phi$ is rational inner,  we obtain commuting unitaries on $K_{\phi}$  that are (often infinite-dimensional) perturbations of the compressed shift operators $K_{\phi}$. We prove that these unitaries are unitarily equivalent to multiplication by the coordinate functions on $L^2(\sigma_\alpha)$ and then establish a number of related properties and simplified results in special cases. 
Finally, we show that the Taylor joint spectrum of these Clark unitaries coincides  with level sets of $\phi$  when $\phi$ is a rational inner function.
\end{abstract}
\maketitle

\section{Introduction and Overview}\label{sec:intro}
\subsection{Introduction}
The shift operator acting on functions in the Hardy space $H^2(\D)$ of the unit disk 
\[\D=\{z\in \C\colon |z|<1\}\]
is a fundamental example of an isometry. 
It is defined on holomorphic functions $f=\sum_{k=0}^{\infty}a_kz^k$ in the disk satisfying $\|f\|^2_{H^2(\D)}=\sum_k|a_k|^2<\infty$ 
as $T\colon f\mapsto z\cdot f$.
The shift operator has been studied extensively over the last century, and many important theorems regarding the Hardy shift have become classical results in operator theory (see, for instance, \cite[Chapters 3 and 4]{GMR}). An important example is a well-known theorem of Beurling which characterizes the invariant subspaces of the shift. Namely, $\mathcal{M}\subset H^2(\D)$ is a non-trivial closed invariant subspace for $T$ if and only if $\mathcal{M}=\phi H^2(\D)$ for some inner function $\phi \colon \D \to \C$. Recall that $\phi$ is inner if $\phi \in H^{\infty}(\D)$, the Banach algebra of bounded holomorphic functions in the disk, and its non-tangential limits satisfy $|\phi^*(\zeta)|=1$ for almost every $\zeta$ on the unit circle \[\T=\{z \in \C\colon |z|=1\}.\]

Another subspace of $H^2(\D)$ associated with an inner function $\phi$ is the model space
\[K_{\phi}=H^2(\D)\ominus \phi H^2(\D)=(\phi H^2(\D))^{\perp}.\]
Acting on each model space is the \emph{compressed shift} $S_{\phi}\colon K_{\phi}\to K_{\phi}$ defined as
\[S_{\phi}=P_{{\phi}}T,\]
where $P_{{\phi}} \colon H^2(\D)\to K_{\phi}$ denotes orthogonal projection. The `model' terminology arises from Sz\H{o}kefalvi-Nagy-Foia\c{s} theory, which realizes a wide class of Hilbert space contractions as compressed shifts on model spaces \cite[Chapter 9]{GMR}.
This fact provides a powerful motivation for studying compressions of the shift acting on model spaces, and also for investigating how properties of the underlying inner function $\phi$ are reflected in its operator theory. An important contribution in this direction was made by D.N. Clark, who observed \cite{Clark} that the rank-one perturbations of $S_{\phi}$ to a unitary operator form a one-parameter family $\{U_{\alpha}\}$ indexed by $\alpha \in \T$. 
He further identified the spectral measures of these unitaries as the positive singular Borel measures $\{\sigma_{\alpha}\}$ that satisfy
\begin{equation}
\frac{1-|\phi(z)|^2}{|\alpha-\phi(z)|^2}=\int_{\T}P_z (\zeta) d\sigma_{\alpha}(\zeta) 
\label{eq:clarkmeasuredef}
\end{equation}
where $P_z(\zeta) = (1-|z|^2)/|\zeta-z|^2$ denotes the Poisson kernel of the disk. These measures are often referred to as \emph{Clark measures }, and their properties are intimately connected, through a unitary map $J_{\alpha} \colon K_{\phi}\to L^2(\sigma_{\alpha})$, with both the function theory of $\phi$ and the spectral theory of $U_{\alpha}$. For instance, $\sigma_{\alpha}$ is the spectral measure of $U_{\alpha}$, and a point in $\mathrm{supp}\,\sigma_{\alpha} \subset \mathbb{T}$ is an eigenvalue of $U_{\alpha}$ precisely if $\phi$ has a finite angular derivative at the corresponding point of the unit circle. We refer the reader to \cite{cimaross,GMR} for accessible introductions to classical Clark theory, and to \cite{Saksman, PoltSara, FryL, LTreil} for a more varied accounting of recent progress (eg. including the case of non-inner $\phi$).

Hardy spaces in the unit disk naturally generalize to the bidisk
\[\D^2=\{(z_1,z_2)\in \C^2\colon |z_j|<1, j=1,2\}\]
and its distinguished boundary $\T^2=\{(\zeta_1,\zeta_2)\in \C^2 \colon |\zeta_j|=1, \,j=1,2\}$.
We say that a holomorphic function $f \colon \D^2\to \C$ belongs to $H^2(\D^2)$ if $f=\sum_{k=0}^{\infty}\sum_{l=0}^{\infty}a_{k,l}z_1^kz_2^l$ satisfies $\|f\|^2_{H^2(\D^2)}=\sum_{k,l}|a_{k,l}|^2<\infty$. Properties of Hardy spaces in the bidisk, and more generally the polydisk, are discussed in \cite{Rud}. In particular, $H^2(\D^2)$ can be identified via Cauchy integrals  with the subspace of $L^2(\T^2)$ whose elements have Fourier series with all nonnegative frequencies \cite[Chapter 3]{Rud}. On $H^2(\D^2)$ we have a pair of shift operators
\[T_1\colon H^2(\D^2)\to H^2(\D^2), \quad (T_1f)(z)=z_1 \cdot f(z_1,z_2)\]
and 
\[T_2\colon H^2(\D^2)\to H^2(\D^2), \quad (T_2f)(z)=z_2\cdot f(z_1,z_2).\]
If $\phi \colon \D^2\to \C$ is an inner function, that is, a bounded holomorphic function on the bidisk with unimodular non-tangential limits on $\T^2$, then the closed subspace $\phi H^2(\D^2)$ is jointly invariant for $(T_1,T_2)$. In stark contrast to the one-variable situation, this does not describe all invariant subspaces (a subspace of $H^2(\D^2)$ is of Beurling type when the shifts satisfy an additional doubly commuting condition \cite{Man}) and the operator theory in $H^2(\D^2)$ is much more complicated. Function theory in the bidisk is similarly complicated: even rational inner functions $\phi=q/p$ need not extend to be continuous on $\T^2$ and can in fact have intricate singularities \cite{BPSII}. Thus, one expects that the structure of two-variable model spaces 
\[K_{\phi}=H^2(\D^2)\ominus \phi H^2(\D^2)\]
and their associated compressed shifts
\begin{equation*}
S^j_{\phi} \colon K_{\phi} \to K_{\phi}, \quad S^j_{\phi}f=P_{\phi}T_jf,
\end{equation*}
(where $P_{{\phi}} \colon H^2(\D^2)\to K_{\phi}$ denotes orthogonal projection) should be fairly complicated, see eg. \cite{BickelLiaw,BhaDas} and references provided there. 

Nevertheless, there has been considerable recent progress on function theory in polydisks and associated operator theory. We do not give a comprehensive overview here, we merely focus on developments that are particularly relevant to the results in the present paper. Model spaces (also called Beurling quotient modules) have recently been characterized in terms of certain operator-theoretic conditions in \cite[Theorem 1.1]{BhaDas}.
Inspired by work of E. Doubtsov \cite{Doubt2020}, Clark measures on $\T^2$ associated with rational inner functions have been studied in great detail \cite{BickelSolaI, BickelSolaII, cal2}, and detailed structure formulas for these measures have been obtained. See \cite{cal1,BergqvistI, jac} for further extensions.

In this paper, we investigate generalizations of Clark theory to two variables by studying commuting Clark unitaries $(U^1_{\alpha}, U^2_{\alpha})$ obtained from compressions of the shifts $(T_1,T_2)$ to two-variable model spaces.\footnote{Terminology varies across the literature, but in this paper, we take `Clark unitary (operator)' to mean a perturbation of a compressed shift on a model space into a unitary operator. We will also encounter the unitary operators that embed model spaces in certain $L^2$ spaces; we will refer to such operators as `Clark embedding operators'. We caution that in two or more variables, these embedding operators are isometric but not always surjective!} Our goals include giving a description of such unitaries, investigating their properties, and showing that, at least in the case of rational inner function (RIF) $\phi$, the Taylor joint spectrum of the unitaries (defined below) is equal to the support 
of the Clark measure $\sigma_{\alpha}$, a one-dimensional subset of the two-torus. Since several methods for constructing RIFs with prescribed properties are known, our results produce a wide range of examples of sets that can arise as Taylor spectra for natural operators acting on spaces of holomorphic functions on the bidisk.

As is to be expected, there are several obstacles we need to overcome along the way. First is the realization that, in contrast to the classical one-dimensional case, we cannot expect finite rank perturbations to be sufficient to turn compressed shifts to unitaries. The perturbations we do obtain involve projections onto the model spaces, which are not always easy to compute explicitly in closed form. The structure of RIF Clark measures is also fairly complicated: these measures act by integrating along unions of smooth curves in $\T^2$ and weighting by non-constant weights obtained from partial derivatives of, in general, globally discontinuous RIFs. Finally, we need to connect the properties of these objects with the structure of the chain complex defining the Taylor joint spectrum.

We should emphasize that we focus on Clark theory in two variables in the setting of the bidisk; see for instance \cite{Jury14, JM18, AleksDoubt2020, AleksDoubt2023} for the setting of the unit ball. There are several reasons for restricting our attention to the case of two variables. On the one hand, the function theory of rational inner functions is more complicated in three or more variables, one manifestation being that the level sets of a RIF in higher dimensions cannot, in general, be parameterized by analytic or even continuous functions \cite{BPSIII}. Our understanding of Clark measures in higher-dimensional polydisks is less detailed (see however \cite{cal2}), with examples of pathological behavior, such as the measures having unbounded weight functions.
On the other hand, the formalism of the Koszul complex \cite{tay} used to define the Taylor joint spectrum becomes more involved as the number of variables increases and well-known operator theoretic obstacles such as the failure of And\^o's inequality \cite{Parrott,var} further support the investigations of pairs of commuting unitaries separately. 

\subsection{Overview}

Our paper is structured as follows. In Section \ref{sec:zwexample}, we study the simple example $\phi=z_1z_2$ and its associated Clark measures and Clark unitary operators. This example is simple enough to be handled from first principles, but the features we observe illustrate the general results that we obtain later on. 

In Section \ref{sec:prelim}, we present background material that we use freely in the remainder of the paper. We first review some salient results from Clark theory in one variable and then move to  a discussion about model spaces (including related Agler decompositions) and Clark measures in two variables, with a special emphasis on the case of rational inner functions. Finally, we recall the definition of Taylor joint spectrum and list some of the features of this notion of multivariate spectrum. 

In Section \ref{sec:unitaries}, we investigate the existence and properties of Clark operators on the bidisk. We start by considering the Clark embedding operators from the model space $K_\phi$ into the Clark space $L^2(\sigma_\alpha)$ and identifying their adjoints. Then in Theorem \ref{thm:U1alpha}, under further assumptions that the embedding operator is unitary and related functions $\abf{1}_\alpha, \abf{2}_\alpha$ are bounded, we prove that the operators $(U_\alpha^1, U_\alpha^2)$ defined on the associated model space by 
\[ 
\begin{aligned}
U^1_\alpha  &= S_\phi^1  + P_\phi P_{H^2_2} M_{\overline{\abf{1}_\alpha}} \\
U^2_\alpha  &= S_\phi^2  + P_\phi P_{H^2_1} M_{\overline{\abf{2}_\alpha}}, 
\end{aligned}
\]
give perturbations of the compressed shifts $(S_\phi^1, S_\phi^2)$ to commuting unitary operators. Here, $M_{\overline{\abf{j}_\alpha}}$ is multiplication by $\overline{\abf{j}_\alpha}$ on $L^2$ and  $P_{H^2_j}$ is the projection from $L^2$ onto a one-variable Hardy space. As in the one-variable situation, we prove that these Clark operators $(U_\alpha^1, U_\alpha^2)$ are unitarily equivalent to multiplication by the coordinate functions on $L^2(\sigma_{\alpha})$. 

The remainder of Section \ref{sec:unitaries} investigates special cases and properties of  $(U_\alpha^1, U_\alpha^2)$. For example, Corollary \ref{cor:crossformula} establishes very simple formulas for these operators in the setting where $\phi$ has a factor of $z_1 z_2.$ Meanwhile, Corollary \ref{cor:RIFUalpha} shows that the assumptions from Theorem \ref{thm:U1alpha} generically hold when $\phi$ is a rational inner function of two variables. Proposition \ref{prop:ptheta} shows that the projection $P_\phi$ is usually needed in the $U_j^\alpha$ formulas and Theorem \ref{thm:perturbissmall} shows that, while these perturbations of compresssed shifts to unitaries are often infinite-dimensional, they are still ``small'' in a way that naturally aligns with the known structure of the two-variable model spaces.  

Then, Section \ref{sec:spec} studies the Taylor joint spectrum of pairs of Clark unitaries. Under the assumption that $\phi$ is a two-variable rational inner function without a monomial factor, Theorem \ref{thm:levsetsasTaylorspec} shows that the Taylor joint spectrum of $(U_\alpha^1, U_\alpha^2)$ is the $\alpha$ level set of $\phi$ on $\mathbb{T}^2$. As part of this analysis, we exploit the precise structure of Clark measures of rational inner functions. 

In Section \ref{sec:examples}, we illustrate our results by applying them to both products of one-variable finite Blaschke products and a singular rational inner function. In both settings, our computations illustrate obstructions to finding tractable ways of explicitly computing $(U_\alpha^1, U_\alpha^2)$ when $\phi$ is not divisible by $z_1 z_2.$

\section{A Simple Example}\label{sec:zwexample}

Consider the rational inner function $\phi(z_1,z_2)=z_1z_2$, a product of two one-variable inner functions, which has $\phi(0,z_2)=\phi(z_1,0)=0$. This example is simple enough to analyze directly (with some minor assistance from \cite{BickelSolaI, BickelSolaII}), and illustrates the results obtained in this paper.

The associated model space $K_{\phi}=H^2(\mathbb{D}^2)\ominus \phi H^2(\mathbb{D}^2)$ consists of functions in $H^2(\mathbb{D}^2)$ whose power series expansions contain no terms $z_1^kz_2^l$ with $k$ and $l$ both positive. Note that we have the orthogonal decomposition
\[K_{\phi}=\overline{\mathrm{span}}\{z_1^k\colon k\geq 0\}\oplus \overline{\mathrm{span}}\{z_2^l \colon l\geq 1\}\]
which can be rewritten as
\[K_{\phi}=H^2_1\oplus z_2H^2_2 = H^2_2\oplus z_1H^2_1,\]
where for $j=1,2$, $H^2_j$ denotes the one-variable Hardy space $H^2(\D)$, embedded in $H^2(\D^2)$ with independent variable $z_j$.
We can readily describe the action of compressed shifts on powers of $z_j$:
\[(S_{\phi}^1z_1^k)(z)=(P_{{\phi}}T_1z_1^k)(z)=(P_{{\phi}}z_1^{k+1})(z)=z_1^{k+1}, \quad k=0,1,\ldots, \]
and 
\[(S_{\phi}^1z_2^l)(z)=(P_{{\phi}}T_1z_2^l)(z)=(P_{{\phi}}z_1z_2^l)(z)=0, \quad l=1,2,\ldots.\]
Similarly, we obtain
\[S_{\phi}^2z_1^k=0, \quad k=1, 2\ldots,\quad \textrm{and} \quad S_{\phi}^2z_2^l=z_2^{l+1}, \quad l=0, 1,2,\ldots.\]
This means that $S^1_{\phi}$ maps the subspace $H^2_1\subseteq K_{\phi}$ into itself and $z_2H^2_2$ onto $\{0\}$; an analogous statement holds for $S^2_{\phi}$. From this, it is clear that, in contrast to the classical one-variable setting, there are no finite-rank perturbations of $S^1_{\phi}$ and $S^2_{\phi}$ that result in unitary operators on $K_{\phi}$. For instance, note that $H^2_2 \subseteq K_\phi$ and $H^2_2 \perp \mathrm{Ran}S^1_{\phi}$. 

However, if we allow infinite-rank perturbations we can write down a concrete representation of unitaries associated with the compressed shifts. Namely, let $\alpha \in \mathbb{T}$ and consider the operators $R_{\alpha}^k\colon K_{\phi}\to K_{\phi}$, $k=1,2$ defined on powers of $z_j$ by $R_{\alpha}^1(z_1^k)=R_{\alpha}^2(z_2^l)=0$ for all $k, l \ge 0$ and 
\[R_{\alpha}^1(z_2^l)=\alpha z_2^{l-1}, \quad l=1,2,\ldots 
\qquad \textrm{and}\quad R_{\alpha}^2(z_1^k)=\alpha z_1^{k-1}, \quad k=1,2,\ldots.\] 
A more compact expression for the $R_{\alpha}^j$ operators, amenable to generalization, is 
\[R_{\alpha}^{1}
=\alpha P_{H_2^2} M_{\bar{\phi}/\bar{z_1}} \ \text{ and } \ \ R_{\alpha}^{2}
=\alpha P_{H^2_1} M_{\bar{\phi}/\bar{z_2}}\]
where $P_{H^2_j}$ is the orthogonal projection from $L^2$ onto $H^2_j$ and $M_{\bar{\phi}/\bar{z_j}}$ is multiplication by the bounded function $\bar{\phi}/\bar{z_j}$ on $L^2$.
Now consider the operators $U^j_{\alpha}\colon K_{\phi}\to K_{\phi}$, $j=1,2$, defined by
\[U^1_{\alpha}=S^1_{\phi}+R^1_{\alpha} \quad \textrm{and}\quad U^2_{\alpha}=S^2_{\phi}+R^2_{\alpha}.\]
Then if $f(z) = f_1(z_1) + z_2 f_2(z_2) \in K_\phi$ for $f_1 \in H^2_1, f_2, \in H^2_2,$ then
\[ (U^1_{\alpha} f)(z) = z_1 f_1(z_1) + \alpha f_2(z_2), \]
and a similar formula holds for $U^2_\alpha.$
From this formula, one can see that $U^1_\alpha$ is unitary and with direct computations, one can easily check that the pair of unitaries 
$(U^1_{\alpha}, U^2_{\alpha})$ commute. 
 Note that the splitting $K_{\phi}=H^2_1 \oplus z_2H^2_2$ is not preserved. For instance, we have $U_{\alpha}^1(z_2)=\alpha\in H_1^2$ while $z_2 \in z_2H^2_2$; this is because $R_{\alpha}^1$ maps $z_2\in z_2H^2_2$ to $\alpha \in H^2_1$. These formulas for $R_\alpha^j$ and $U_\alpha^j$ align with the more general perturbations of compressed shifts to unitaries that will appear later in Theorem \ref{thm:U1alpha}  and Corollary \ref{cor:crossformula}.

\begin{figure} 
\includegraphics[width=0.45 \textwidth]{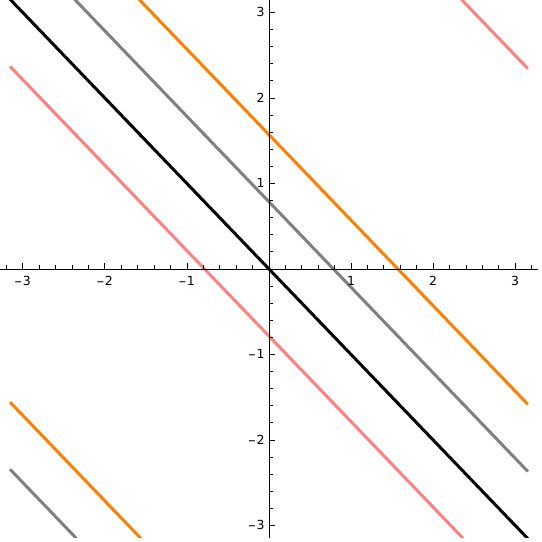}
\caption{Support sets  in $\mathbb{T}^2\simeq [-\pi, \pi)^2$ for the Clark measures $\sigma_{\alpha}$ associated with $\phi=z_1z_2$ and $\alpha=1$ (black), $\alpha=e^{i\frac{\pi}{4}}$ (gray), $\alpha=e^{i\frac{\pi}{2}}$ (orange), and $\alpha=e^{-i\frac{\pi}{4}}$ (pink).}
\label{fig:zwsupport1}
\end{figure}

The family $\{\sigma_{\alpha}\}_{\alpha\in \mathbb{T}}$ of Clark measures associated with $\phi$ (formally defined in Subsection \ref{subsec:RIFClark}) can readily be identified, see e.g. \cite[Example 4.2]{Doubt2020}. Namely, we have the pairing
\[\int_{\mathbb{T}^2}f (\zeta) d\sigma_{\alpha}(\zeta) =\int_{\mathbb{T}}f(\zeta_1, \alpha \bar{\zeta}_1)|d\zeta_1|, \quad f\in C(\mathbb{T}^2).\]
This shows $\mathrm{supp}\,\sigma_{\alpha}=\{(\zeta_1, \alpha \bar{\zeta}_1)\colon \zeta_1 \in \mathbb{T}\}$. A collection of these support sets is graphed in Figure \ref{fig:zwsupport1}.

As is explained in \cite{Doubt2020}, there is a surjective isometry $J_{\alpha}\colon K_{\phi}\to L^2(\sigma_{\alpha})$ for each value of $\alpha$. By \cite[Theorem 1.3]{BickelSolaI}, this isometry acts by evaluation of $f\in K_{\phi}$ on $\mathrm{supp}\,\sigma_{\alpha}$, which is well-defined for $\sigma_{\alpha}$-a.e. $\zeta\in \T^2$. (In this example, one could also check directly that this evaluation map gives a unitary operator from $K_{\phi}\to L^2(\sigma_{\alpha})$.) Then, for almost every $ \zeta \in \mathrm{supp}\,\sigma_{\alpha}$ and each power of $z_j,$
\[
\begin{aligned} (J_{\alpha}U^1_{\alpha}z_1^k)(\zeta) &=\zeta_1^{k+1} = \zeta_1 \cdot \zeta_1^k  = (M_{\zeta_1} J_\alpha z_1^k)(\zeta)   \\(J_{\alpha}U^1_{\alpha}z_2^l)(\zeta) &= \alpha (\alpha \bar{\zeta_1})^{l-1} = 
\zeta_1 \cdot (\alpha \bar{\zeta}_1)^{l} = (M_{\zeta_1} J_\alpha z_2^l)(\zeta), \end{aligned} \]
so that $J_{\alpha}U^1_{\alpha}=M_{\zeta_1} J_\alpha$, where $M_{\zeta_1}\colon L^2(\sigma_{\alpha})\to L^2(\sigma_{\alpha})$ is the unitary furnished by multiplication by the coordinate function $\zeta_1$. Similarly, 
\[(J_{\alpha}U^2_{\alpha})(z_1^k)=\alpha \bar{\zeta}_1 \cdot \zeta_1^k \quad \textrm{and} \quad (J_{\alpha}U^2_{\alpha})(z_2^l)=\alpha \bar{\zeta}_1 \cdot (\alpha \bar{\zeta}_1)^{l},\]
which implies that $J_{\alpha}U^2_{\alpha}=M_{\alpha \bar{\zeta}_1} J_\alpha =M_{\zeta_2} J_\alpha$.

Finally, one can check directly (or using Theorem \ref{thm:jointspectrum} below) that the Taylor joint spectrum of $(U^1_{\alpha}, U^2_{\alpha})$ (see Subsection \ref{subsec:Tspec} for a definition) is
\[\tay(U^1_{\alpha}, U^2_{\alpha})=\{(\zeta_1, \alpha \bar{\zeta}_1)\colon\zeta_1 \in \mathbb{T}\},\]
so that each pair of unitaries has Taylor spectrum equal to a straight line in the torus. Finally, note that $\bigcup_{\alpha \in \mathbb{T}}\tay(U^1_{\alpha}, U^2_{\alpha})=\mathbb{T}^2$ as a disjoint union.

\section{Preliminaries}\label{sec:prelim}
We review some basic facts from classical Clark theory in the unit disk, as well as more recent work concerning rational inner functions in two variables and their Clark measures. Then, we review the definition and the basic properties of the Taylor joint spectrum of a pair of commuting operators.
\subsection{Clark Theory in One Variable}\label{subsec:1dClark}
For comprehensive introductions to Clark theory, we refer the reader to \cite{cimaross, GMR, Saksman} and their lists of references. 

Recall that each inner function $\phi\colon \mathbb{D}\to \mathbb{C}$ gives rise to a family of Clark measures $\{\sigma_{\alpha}\}_{\alpha \in \mathbb{T}}$ via \eqref{eq:clarkmeasuredef}; each $\sigma_{\alpha}$ is singular with respect to Lebesgue measure on the unit circle, and moreover $\sigma_{\alpha_1}\perp
 \sigma_{\alpha_2}$ for distinct $\alpha_j$. Clark measures can be identified as spectral measures associated with certain unitary operators acting on the model space $K_{\phi}=H^2(\mathbb{D})\ominus \phi H^2(\mathbb{D})$. Namely, while the compressed shift $S_{\phi}\colon K_{\phi} \to K_{\phi}$ is not generally unitary (see eg. \cite[Chapter 9]{GMR}), they do admit one-dimensional perturbations to unitaries. More specifically, 
for each $\alpha \in \mathbb{T}$, one can define the unitary operator $U_\alpha$ by
\begin{equation}\label{eq:clarkUdef}
U_{\alpha}f:=S_{\phi}f+\frac{\alpha}{1-\overline{\phi(0)}\alpha} \langle f, Ck^\phi_0\rangle k^\phi_0, \quad f\in K_{\phi},
\end{equation}
where $k^\phi_0$ is the reproducing kernel for $K_{\phi}$ at the origin, and $Ck^\phi_0$ is its conjugation. If $\phi(0)=0$, then $k^\phi_0=1$ and \eqref{eq:clarkUdef} simplifies to
\[U_{\alpha}f=S_{\phi}f+\alpha \langle f,T^*\phi\rangle 1,\]
where $T^*$ denotes the backwards shift.

Define the surjective isometry $J_{\alpha}\colon K_{\phi}\to L^2(\sigma_{\alpha})$ by setting \[(J_{\alpha}k^{\phi}_{w})(z)=\frac{1-\alpha\overline{\phi(w)}}{1-\bar{w}z}\] 
for the reproducing kernel functions
\[k^{\phi}_{w}(z)=\frac{1-\overline{\phi(w)}\phi(z)}{1-\bar{w}z}, \quad z\in \mathbb{D},\]
and then extending to $K_{\phi}$ by linearity. 
Then the following theorem is well known. 
\begin{thm}
For each $\alpha$, the operator $U_{\alpha}$ acting on the model space $K_{\phi}$ is unitarily equivalent under $J_{\alpha}$ to the operator $M_{\zeta}$ on the spectral representation $L^2(\sigma_{\alpha})$.
\end{thm}
See \cite[Chapter 8, Chapter 11]{GMR} for a detailed exposition, including proofs of the above statements.  

Much of this paper is concerned with rational inner functions. This class generalizes finite Blaschke products in one variable, whose Clark theory is well understood.
\begin{example}
Let $\Lambda=\{\lambda_1,\ldots, \lambda_n\}\subset \mathbb{D}$ and let
\[\phi(z)=e^{ia}z^m\prod_{j=1}^{n}b_{\lambda_j}(z), \quad b_{\lambda_j}(z)=\frac{|\lambda_j|}{\lambda_j}\frac{\lambda_j-z}{1-\overline{\lambda_j}z},\] 
be a finite Blaschke product with possibly repeated zeros at $\Lambda \cup \{0\}$. Note that each such $\phi$ extends to be analytic past $\T$.

Then the associated model space $K_{\phi}$ has $\dim K_{\phi}=\deg \phi$ (see \cite[Proposition 5.16]{GMR}), hence is finite-dimensional, and is spanned by reproducing kernels and their derivatives. In fact, $K_{\phi}$ is finite dimensional precisely when $\phi$ is a finite Blaschke product.

The Clark measures associated with a finite Blaschke product admit the explicit representation 
\begin{equation}
\sigma_{\alpha}=\sum_{\zeta \in \mathcal{C}_{\alpha}}\frac{1}{|\phi'(\zeta)|}\delta_{\zeta},
\label{eq:clarkmeasblaschke}
\end{equation}
where $\delta_{\zeta}$ denotes point evaluation at $\zeta$ and $\mathcal{C}_{\alpha}=\{\zeta \in \mathbb{T}\colon \phi(\zeta)=\alpha\}$ is a unimodular level set of $\phi$; see \cite[Proposition 11.2]{GMR}.
The expression in \eqref{eq:clarkmeasblaschke} is a finite sum since $\phi$ is a finite cover of $\mathbb{D}$, and in particular $\#\mathcal{C}_{\alpha}=\deg \phi$. The elements of $\mathcal{C}_{\alpha}$ are the eigenvalues of $U_{\alpha}$ and, in particular, the spectrum of the unitary $U_{\alpha}$ coincides precisely with the support of $\sigma_{\alpha}$, so that $\sigma(U_{\alpha})=\{\zeta \in \mathbb{T}\colon \phi(\zeta)=\alpha\}$.
\end{example}

\subsection{Clark Measures for Two-Variable Inner Functions}\label{subsec:RIFClark}
We refer the reader to \cite{Rud,BPSII,BKPS} for background material concerning rational inner functions. Discussions of their associated Clark measures can be found in \cite{Doubt2020,BickelSolaI,BickelSolaII, BergqvistII}. More general results have recently been obtained in \cite{cal1,cal2}.

A function $\phi \in H^{\infty}(\D^2)$ is said to be inner if its non-tangential limit $\phi^*(\zeta)=\angle \lim_{z\to \zeta}\phi(z)$, which exists a.e. \cite[Theorem 3.3.5]{Rud}, satisfies that $|\phi^*(\zeta)|=1$ at almost every $\T^2$. We say that $\phi$ is a rational inner function if $\phi=q/p$ for $q,p\in \C[z_1, z_2]$, where $Z_p\cap \D^2=\emptyset$, and $q$ and $p$ are relatively prime. Polynomials with no zeros in $\D^2$ are called \emph{stable}; note that $Z_p\cap \T^2\neq \emptyset$ is not ruled out, and in fact the case of boundary zeros is of considerable interest.
One can show (see \cite[Theorem 5.2.5]{Rud}) that $\phi$ is a rational inner function in $\D^2$ if and only if
\[\phi(z)=e^{ia}z^m\frac{\tilde{p}(z)}{p(z)},\]
where $a\in \R$, $m=(m_1,m_2)$ is a multi-index, and $p$ is a stable polynomial. The \emph{reflection} of $p$ is defined as
\[\tilde{p}(z)=z_1^{n_1}z_2^{n_2}\overline{p\left(\frac{1}{\bar{z}_1}, \frac{1}{\bar{z}_2}\right)},\]
with $(n_1,n_2)$, the bidegree of $p$, given by $n_j=\deg_{z_j}p$ for $j=1,2$. 

Let us assume that $p$ is \emph{atoral}, meaning in our context that $p$ and its reflection polynomial do not share a common factor (see \cite{AMS} for a general discussion).  If $p$ is stable and atoral, then $p$ also has no zeros on the faces of the bidisk $\mathbb{D}\times \mathbb{T}$ and $\mathbb{T}\times \mathbb{D}$, see \cite[Lemma 10.1]{Kne15}.  In contrast to the one-variable setting, a rational inner function (RIF) in two variables can have singularities on $\T^2$: these occur at common zeros in $Z_p\cap Z_{\tilde{p}}\cap \T^2$. By a theorem of Knese (\cite[Theorem C]{Kne15}; see also \cite[Theorem 3.2]{BKPS} for an extension to general bounded rational functions), any RIF has non-tangential limits at \emph{every} point of $\zeta \in \T^2$, and these non-tangential limits satisfy $\phi^*(\zeta) \in \T$. It should be noted that the function $\zeta \mapsto \phi^*(\zeta)$ is discontinuous on $\T^2$ whenever $Z_p \cap \T^2 \neq \emptyset$ (see for instance \cite{PasBLMS}).

The \emph{model space} associated with an inner function $\phi$ is defined as 
\[K_{\phi}=H^2(\D^2)\ominus \phi H^2(\D^2).\]
The structure of $K_{\phi}$ in two variables is in general more complicated than in one variable, even if $\phi$ is rational inner; see \cite{BickelLiaw} and the references listed there. 
Associated with these spaces are the Szeg\H{o} kernel for the bidisk and the associated reproducing kernel
\[
\begin{aligned}
k_w(z) &= \frac{1}{(1-z_1 \bar{w}_1)(1-z_2 \bar{w}_2)},  \\
k^\phi_w(z) &=  \frac{1-\phi(z) \overline{\phi(w)}}{(1-z_1 \bar{w}_1)(1-z_2 \bar{w}_2)}.  
\end{aligned}
\]
so that $H^2(\mathbb{D}^2) = \mathcal{H}(k_w(z))$ and $K_{\phi}  = \mathcal{H}( k^{\phi}_w(z) )$ as reproducing kernel Hilbert spaces. Finally, recall that the compressions $S^j_{\phi}\colon  K_{\phi}\to K_{\phi}$ of the shifts $(T_1,T_2)$ are defined by
\begin{equation}\label{def:comprshifts}
(S^j_{\phi}f)(z)=(P_{{\phi}}T_j f)(z), \quad z\in \D^2,
\end{equation}
where $P_{{\phi}}\colon H^2(\D^2)\to K_{\phi}$ denotes the orthogonal projection of $H^2(\D^2)$ onto its closed subspace $(\phi H^2(\D^2))^{\perp}$. As is proved in \cite[Theorem 1.1]{BhaDas}, a closed subspace $\mathcal{Q}\subset H^2(\mathbb{D}^2)$ invariant under $(T_1^*, T_2^*)$ is in fact of the form $\mathcal{Q}=K_{\phi}$ for some inner function $\phi$ precisely if
\[(I-(P_{\mathcal{Q}}T_j)^*P_{\mathcal{Q}}T_j)(I-P_{\mathcal{Q}}T_k(P_{\mathcal{Q}}T_k)^*)=0, \quad j\neq k,\]
where $I$ is the identity operator and $P_{\mathcal{Q}}\colon H^2(\mathbb{D}^2)\to \mathcal{Q}$ denotes orthogonal projection; this result remains valid in $d$ variables. Bhattacharjee, Das, Debnath, and Sarkar further provide a characterization of $d$-tuples of commuting operators which are unitarily equivalent to compressions of the shift to a model space \cite[Theorem 3.2]{BhaDas}. Thus, many of our results apply to pairs of operators covered by this theorem.
 
We turn to a discussion of Clark measures associated with a rational inner function $\phi$. For a fixed $\alpha \in \T$, define the \emph{unimodular level set} 
\begin{align}\label{e-Calpha}\mathcal{C}_{\alpha}=\overline{\{\zeta \in \T^2\colon \phi^{\ast}(\zeta)=\alpha\}}.\end{align}
As is explained in \cite{BPSII},  when $\phi = \tfrac{\tilde{p}}{p}$,  we have the identification $\mathcal{C}_{\alpha}=\{\zeta \in \T^2\colon \tilde{p}(\zeta)-\alpha p(\zeta)=0\}$. A fortiori, the level set $\mathcal{C}_{\alpha}$ can be parameterized by analytic functions \cite[Theorem 2.8]{BPSII} (see also the discussion in \cite[Section 2]{BickelSolaII}, and in particular Lemma 2.5 therein).  

Clark measures on the bidisk are defined in a similar manner as for the unit disk. If $\phi \colon \D^2\to \C$ is an inner function, then for each $\alpha\in \T$
\[\frac{1-|\phi(z)|^2}{|\alpha-\phi(z)|^2}=\Re\left(\frac{\alpha+\phi(z)}{\alpha-\phi(z)}\right), \quad z\in \D^2\]
is a positive pluriharmonic function, and hence
\[\frac{1-|\phi(z)|^2}{|\alpha-\phi(z)|^2}=\int_{\T^2}P_{z_1}(\zeta_1)P_{z_2}(\zeta_2)d\sigma_{\alpha}(\zeta)\]
for some positive singular Borel measure $\sigma_\alpha$ on $\T^2$ and for one-variable Poisson kernels $P_{z_j}$. These measures $\{\sigma_{\alpha}\}$ are the Clark measures associated with $\phi$. Some basic properties of Clark measures in polydisks were established by Doubtsov in \cite{Doubt2020}; for instance, we have $\sigma_{\alpha_1}\perp \sigma_{\alpha_2}$ whenever $\alpha_1\neq\alpha_2$, as in one variable, and an analog of the Aleksandrov disintegration theorem holds for $\{\sigma_{\alpha}\}$ (see \cite[Proposition 2.4]{Doubt2020}). 

These measures admit a particularly concrete representation for $\phi$ rational inner. For each fixed $\zeta_1 \in \T$, the function $z_2\mapsto \phi(\zeta_1, z_2)$ is a finite Blaschke product, with degree at most $n_2=\deg_{z_2}p$. The formula \eqref{eq:clarkmeasblaschke} describes the Clark measures associated with a finite Blascke product, and from this one might hope to obtain a corresponding formula for two-variable RIFs by slicing. This is indeed possible, and is one of the main results of \cite{BickelSolaII}. Before we formulate this as a theorem, we need to introduce the notion of exceptional $\alpha \in \T$. We say that $\alpha\in \T$ is \emph{generic} if $\mathcal{C}_{\alpha}$ does not contain any sets of the form $\{(\zeta, \eta)\colon \zeta \in \T\}$ or $\{(\eta, \zeta)\colon \zeta \in \T\}$ for some $\eta \in \T$. In other words, $\alpha$ is generic for $\phi$ if the corresponding level set does not contain any horizontal or vertical lines. If $\alpha$ fails to be generic, then we say that $\alpha$ is \emph{exceptional}. Note that a RIF  of two variables  can have at most a finite number of exceptional $\alpha$ (see \cite{BickelSolaII}).

\begin{thm}[\cite{BickelSolaII}, Theorem 3.3]\label{theorem:RIFclarkmeas}
Let $\phi=\frac{\tilde{p}}{p}$ be a bidegree $(n_1,n_2)$ RIF with $n_1,n_2\geq 1$,  and let $\alpha \in \T$ be generic for $\phi$. Then for $f\in C(\T^2)$, the associated Clark measure $\sigma_{\alpha}$ satisfies
\[\int_{\T^2}f(\zeta)d\sigma_{\alpha}(\zeta)=\sum_{j=1}^{n_{2}}\int_{\T}f(\xi, g_j^{\alpha}(\xi))\frac{dm(\xi)}{|\frac{\partial \phi}{\partial z_2}(\xi, g^{\alpha}_j(\xi))|},\]
where $g^{\alpha}_1,\ldots, g^{\alpha}_{n_2}$ are analytic functions parameterizing $\mathcal{C}_{\alpha}$.
\end{thm}
Here, and throughout, $dm$ denotes normalized Lebesgue measure on $\mathbb{T}$.
Note the similarity with \eqref{eq:clarkmeasblaschke}: we sum along level sets, with weights given by derivatives of the underlying rational inner functions.  It is worth pointing out that, by results in \cite{BergqvistI}, Clark measures in two variables satisfy $\dim (\mathrm{supp} \sigma_{\alpha})\geq 1$ so that the smooth curves appearing in our RIF Clark measures are in a certain sense minimal, just like the point masses associated with finite Blaschke products in the unit disk. In contrast to the one-variable situation, however, $\frac{\partial \phi}{\partial z_2}$ need not be well-behaved on $\T^2$, but the weights appearing on the right-hand side in Theorem \ref{theorem:RIFclarkmeas} are nevertheless integrable.

For each $\alpha \in \mathbb{T}$, let  the \emph{Clark embedding} $J_\alpha: K_{\phi} \rightarrow L^2(\sigma_\alpha)$ be initially defined on the reproducing kernels $k^{\phi}_w$ by 
\begin{equation} \label{eqn:Jalpha} (J_\alpha k^{\phi}_w)(z)
= \frac{1-\alpha \overline{\phi(w)}}{(1-z_1 \bar{w}_1)  (1-z_2 \bar{w}_2)} = (1-\alpha \overline{\phi(w)}) k_w(z),  \end{equation}
for all $w \in \mathbb{D}^2$. Then the work of Doubtsov in \cite[Theorem 3.1]{Doubt2020} shows that $J_\alpha$ extends to a linear isometry on $K_{\phi}$. In two or more variables, this isometry need not be surjective (viz. \cite{Doubt2020} and \cite{BickelSolaI}), but fortunately, we have the following.

\begin{thm}[\cite{BickelSolaII}, Theorem 4.2]\label{theorem:poltoratski}
Let $\phi$ be a bidegree $(n_1,n_2)$ RIF  with $n_1,n_2\geq 1$ and let $\alpha \in \T$ be a generic value for $\phi$. Then the Clark embedding $J_{\alpha} \colon K_{\phi}\to L^2(\sigma_{\alpha})$ is unitary.
\end{thm}

As a note, even though Theorem 4.2 is proved in \cite{BickelSolaII} for RIFs of the form $\frac{\tilde{p}}{p}$, the proof extends without issue to RIFs with an additional monomial factor. 
Corollary \cite[Corollary 5.7]{BickelSolaII} shows that for all but possibly finitely many $\alpha$, the densities appearing in the right-hand side of Theorem \ref{theorem:RIFclarkmeas} are bounded functions on $\mathbb{T}$ which are non-zero Lebesgue almost everywhere. This implies that $L^2(\sigma_{\alpha})$ is infinite-dimensional. Together with Theorem \ref{theorem:poltoratski}, this implies the following fact, which is in contrast to the case of a one-variable finite Blaschke product.
\begin{cor}\label{cor:infdimK}
Let $\phi=\frac{\tilde{p}}{p}$ be a bidegree $(n_1,n_2)$ RIF with $n_1,n_2\geq 1$. Then $K_{\phi}$ is infinite-dimensional.
\end{cor}

\subsection{Agler decompositions}\label{aglerdec}
The conclusion in Corollary \ref{cor:infdimK} can also be deduced from other known results. For example, the existence of Agler decompositions for Schur functions on the bidisk  (see \cite{Agler90, bsv05})   immediately implies that for any nonconstant inner function $\phi$ on $\mathbb{D}^2$, the associated model space $K_\phi$ is infinite-dimensional. 

  As is explained in \cite{BickelLiaw, BickelGorkin17}, the model space $K_{\phi}$ can be decomposed as
\begin{equation}
K_{\phi}=S_1^{\mathrm{max}}\oplus S^{\mathrm{min}}_2,
\label{maxmindecomp}
\end{equation}
where the subspace $S_1^{\mathrm{max}}$ is the largest subspace of $K_\phi$ invariant with respect to $T_1$. It is also true (but non-trivial to prove) that $S_2^{\mathrm{min}}$ is $T_2$-invariant. Moreover, we have
\[S_1^{\mathrm{max}}=\mathcal{H}\left(\frac{K_1(z,w)}{1-\bar{w}_1z_1}\right) \ \ \text{ and }  \ \ S_2^{\mathrm{min}}=\mathcal{H}\left(\frac{K_2(z,w)}{1-\bar{w}_2z_2}\right),\]
for positive kernels $K_1, K_2: \mathbb{D}^2 \times \mathbb{D}^2 \rightarrow \mathbb{C}$, where
\[ \mathcal{H}(K_1 )=S_1^{\mathrm{max}} \ominus T_1 S_1^{\mathrm{max}} \ \ \text{ and } \ \ \mathcal{H}(K_2)=S_2^{\mathrm{min}} \ominus T_2 S_2^{\mathrm{min}}  \]
and thus, 
\[ S_1^{\mathrm{max}} = \bigoplus_{k=0}^{\infty}  z_1^k \mathcal{H}(K_1 )\ \ \text{ and } \ \ S_2^{\mathrm{min}} = \bigoplus_{l=0}^{\infty}  z_2^l \mathcal{H}(K_2 ).\]

We also recall (see \cite[Theorem 3.2]{BickelLiaw}) that if a RIF $\phi$ satisfies
$\deg_{z_1}\phi=n_1$ and $\deg_{z_2}\phi=n_2$, then 
\begin{equation}
\dim(\mathcal{H}(K_1))=n_2\quad \textrm{and} \quad \dim(\mathcal{H}(K_2))=n_1.
\label{kerneldims}
\end{equation}

\subsection{Taylor Joint Spectrum}\label{subsec:Tspec}
While there are several different notions of spectrum for tuples of commuting bounded operators, the Taylor joint spectrum has proven to be a particularly useful tool in the study of shift operators on Hilbert spaces of analytic functions; for instances of this, see \cite{CT,GRS,Bhatta} and their lists of references. Comprehensive treatments of the Taylor joint spectrum for tuples of operators can be found in \cite{tay,cur} or \cite[Chapter IV]{mul}. These references also include an overview of other spectral notions (the Harte spectrum, etc.) and conditions for when they coincide. Since we deal exclusively with pairs of operators, we only provide the relevant definitions in that setting. 

Let $(T_1, T_2)$ be a pair of commuting bounded operators on a Hilbert space $\mathcal{H}$. The \emph{Taylor joint spectrum} of $(T_1, T_2)$, denoted $\tay(T_1, T_2)$, is defined as follows. For each $(\lambda_1, \lambda_2) \in \mathbb{C}^2$, consider the associated \emph{Koszul complex} $K(T_1-\lambda_1, T_2-\lambda_2)$:
\[ 0 \xrightarrow{\delta_0} \mathcal{H} \xrightarrow{\delta_1} \mathcal{H} \oplus \mathcal{H} \xrightarrow{\delta_2} \mathcal{H} \xrightarrow{\delta_3} 0,\]
where $\delta_0(0) =0$, 
\[ 
\begin{aligned}
\delta_1(h) &= ( (T_1-\lambda_1) h , (T_2-\lambda_2) h), \quad  \text{ for all } h \in \mathcal{H}, \\ 
\delta_2(h_1, h_2)&  = (T_1 -\lambda_1) h_2 - (T_2-\lambda_2) h_1, \text{ for all } h_1, h_2 \in \mathcal{H}
\end{aligned}
\]
and $\delta_3(h) = 0$ for all $h \in \mathcal{H}$. By construction, the $\delta_n$ are bounded linear maps and $\text{Ran}( \delta_{n-1} )\subseteq \text{Ker}( \delta_n)$ for each $n=1,2,3$ so that $K(T_1-\lambda_1, T_2-\lambda_2)$ is indeed a chain complex.

If we also have $ \text{Ker}( \delta_n)  \subseteq \text{Ran} (\delta_{n-1})$ for all $n$, then we say $K(T_1-\lambda_1, T_2-\lambda_2)$ is \emph{nonsingular} and $(\lambda_1, \lambda_2)  \not \in \tay(T_1, T_2)$.  Meanwhile, if $ \text{Ker}( \delta_n)  \not \subseteq \text{Ran} (\delta_{n-1})$ for any $n$, then we say $K(T_1-\lambda_1, T_2-\lambda_2)$ is \emph{singular} and $(\lambda_1, \lambda_2) \in \tay(T_1, T_2)$. Note that $\textrm{Ker}(\delta_1)\neq\{0\}=\mathrm{Ran}(\delta_0)$ amounts to the existence of a joint eigenvector with associated joint eigenvalue $(\lambda_1,\lambda_2) \in \C^2$, while $\mathrm{Ran}(\delta_2)\neq \mathrm{Ker}(\delta_3)$ means that the joint range of the pair $(T_1-\lambda_1, T_2-\lambda_2)$ fails to equal $\mathcal{H}$.

Let us list some basic properties that are particularly relevant for our work; other distinguishing features of the Taylor joint spectrum, such as the existence of a holomorphic functional calculus, are discussed in \cite{tay,cur,mul}. 
\begin{prop}\label{prop:taybasics}
Let $(T_1,T_2)$ be a pair of commuting bounded operators on a Hilbert space $\mathcal{H}$. Then
\begin{enumerate}
\item The set $\tay(T_1,T_2)$ is a compact and non-empty subset of $\mathbb{C}^2$. 
\item We have $\tay(T_1, T_2)\subset \sigma(T_1) \times \sigma(T_2)$, where $\sigma(T_j)$ denotes the spectrum of the single operator $T_j$.
\item If the pair $(T_1,T_2)$ is jointly similar to the pair $(T'_1,T'_2)$ then \[\tay(T_1,T_2)=\tay(T'_1,T'_2).\]
\end{enumerate}
\end{prop}
In particular, if $(T_1,T_2)$ is a pair of commuting unitaries, as will be the case in this paper, then (ii) implies that $\tay(T_1,T_2)\subset \T^2$. 

\begin{example}[\cite{CT}, Theorem 5]
Let $(T_1,T_2)$ be the shift operators acting on $H^2(\D^2)$. Then $\tay(T_1,T_2)=\overline{\D^2}$.
\end{example} 
Further examples of specific sets arising as Taylor joint spectra can be found, for instance, in \cite{BuPa}.

\section{Clark Unitary Operators $(U_\alpha^1, U_\alpha^2)$}\label{sec:unitaries}
Let $\phi$ be a nonconstant inner function in $\D^2$. We begin by computing the adjoint of the associated Clark embedding $J_{\alpha}$ which was defined in Subsection \ref{subsec:RIFClark}. The following result is a direct analog of Clark's formula for the one-variable adjoint, see e.g., \cite[Theorem 11.6]{GMR}.

\begin{lm} The adjoint $J_\alpha^*:  L^2(\sigma_\alpha) \rightarrow  K_\phi$ of $J_\alpha$ from \eqref{eqn:Jalpha} is given by 
\[ (J_\alpha^*h)(z) = (1-\bar{\alpha} \phi(z)) \int_{\mathbb{T}^2} \ h(\zeta) \ k_\zeta(z) \ d\sigma_\alpha(\zeta),\]
for  $h \in L^2(\sigma_\alpha)$ and $z \in \mathbb{D}^2$. 
\end{lm}

\begin{proof} Fix $h \in L^2(\sigma_\alpha)$ and $z \in \mathbb{D}^2$.  Observe that
\[
\begin{aligned}
(J_\alpha^*h)(z) & = \left \langle J_\alpha^*h, k^\phi_z \right \rangle_{K_\phi} \\
& =  \left \langle h, J_\alpha k^\phi_z \right \rangle_{L^2(\sigma_\alpha)} \\
& =  \left \langle h,  (1-\alpha \overline{\phi(z)}) k_z \right \rangle_{L^2(\sigma_\alpha)} \\
& = (1-\bar{\alpha}\phi(z)) \int_{\mathbb{T}^2} h(\zeta)\ k_\zeta(z) \ d \sigma_\alpha(\zeta),
\end{aligned}
\] 
as needed. 
\end{proof}

Our next goal is to use the operators $J_\alpha$ and $J_\alpha^*$ to connect the multiplication operators $M_{z_1}$ and $ M_{z_2}$ on $L^2(\sigma_\alpha)$ with the compressed shifts $S_\phi^1$ and $S_\phi^2$ on $K_\phi$.
To do that, we need to recall the backward shift $B_1$ in $z_1$, which can be applied to any  function $f$ holomorphic on $\mathbb{D}^2$ as follows:
\[ (B_1 f)(z) = \frac{ f(z) - f(0,z_2)}{z_1}, \quad z=(z_1,z_2)\in \D^2.\]
The backward shift satisfies the following product formula 
\begin{equation} \label{eqn:BWSproduct} \big(B_1 (fg) \big)(z) = f(z) (B_1 g)(z) + g(0, z_2) (B_1 f )(z),\end{equation}
for all $f,g$ holomorphic on $\mathbb{D}^2$. Also for $\alpha \in \mathbb{T}$, define the function $\abf{1}_\alpha$, holomorphic on $\mathbb{D}^2$, by
\begin{equation} \label{eqn:psi_alpha}
    \abf{1}_\alpha(z):= \bar{\alpha} \frac{(B_1\phi)(z)}{1- \bar{\alpha}\phi(0, z_2)}.
\end{equation} 
Then we have the following intertwining result:

\begin{prop} \label{prop:V1alpha} On $L^2(\sigma_\alpha)$, we have $J_\alpha^* M_{\bar{z}_1} = V^1_\alpha J_\alpha^*$, where $V^1_\alpha$ is defined on the range of $J_\alpha^*$ in $K_\phi$ by 
\[ (V_\alpha^1 f)(z) = (B_1 f)(z) + \abf{1}_\alpha(z) f(0, z_2), \text{ for all } z\in \mathbb{D}^2.\]

\end{prop} 

\begin{proof}  Fix $g\in L^2(\sigma_\alpha)$ and define $G(z) = \langle g, k_z \rangle_{L^2(\sigma_\alpha)}$ for $z \in \mathbb{D}^2$. We have 
\[ 
\begin{aligned}
(J_\alpha^* M_{\bar{z}_1} g )(z) &= (1-\bar{\alpha}\phi(z)) \int_{\mathbb{T}^2} \bar\zeta_1 g(\zeta) \ k_\zeta(z) \ d \sigma_\alpha(\zeta) \\
& = (1-\bar{\alpha}\phi(z)) (B_1 G) (z) \\
& = B_1(  (1-\bar{\alpha}\phi) G)(z) -( B_1 (1-\bar{\alpha}\phi))(z) G(0, z_2), 
\end{aligned}
\]
by the identity $\bar{\zeta}_1k_{\zeta}=\frac{1}{z_1}(k_{\zeta}-k_{(0,\zeta_2)})$ and the product formula \eqref{eqn:BWSproduct}. Then, by the definition of $J_\alpha^*$ and $G$, we have 
\[ 
\begin{aligned}
(J_\alpha^* M_{\bar{z}_1} g )(z) & = (B_1J^*_\alpha g)(z) + \bar{\alpha} (B_1 \phi)(z) \left \langle g, k_{(0, z_2)} \right \rangle_{L^2(\sigma_\alpha)} \\
&= (B_1J^*_\alpha g)(z) + \bar{\alpha} (B_1 \phi)(z) \left \langle g, \frac{1}{1-\alpha \overline{\phi(0, z_2)}} J_\alpha k^\phi_{(0, z_2)} \right \rangle_{L^2(\sigma_\alpha)} \\
& = (B_1J^*_\alpha g)(z) + \abf{1}_\alpha(z)  \left \langle J_\alpha^*g, k^\phi_{(0, z_2)} \right \rangle_{K_\phi} \\
& = (B_1J^*_\alpha g)(z) + \abf{1}_\alpha(z) (J_\alpha^*g)(0, z_2), \end{aligned}
\] 
where we used the fact that $J_\alpha k^\phi_{(0, z_2)} = (1- \alpha \overline{\phi(0, z_2)}) k_{(0, z_2)}$ and the reproducing property of $k^\phi_{(0, z_2)}$. 
\end{proof}

Now we will use Proposition \ref{prop:V1alpha} to connect $M_{z_1}$ on $L^2(\sigma_\alpha)$ with the compressed shift $S_\phi^1$ on $K_\phi$ for certain $\phi$. To do that, recall that $H^2_2$ is the Hardy space on the disk with independent variable $z_2$. If $f \in H^2(\mathbb{D}^2)$, then
\[ (P_{H^2_2} f)(z) = f(0, z_2).\]

Then we have the following theorem:

\begin{thm} \label{thm:U1alpha} Let $J_\alpha: K_\phi \rightarrow L^2(\sigma_\alpha)$ be surjective and let $\abf{1}_\alpha$ be bounded on $\mathbb{D}^2$. Then 
$M_{z_1}$ on $L^2(\sigma_\alpha)$ satisfies $M_{z_1}  = J_\alpha U^1_\alpha J^*_\alpha$, where  $U^1_\alpha$ is a unitary operator on $K_\phi$ defined by
\begin{equation} \label{eqn:Ualpha1} U^1_\alpha  = S_\phi^1  + P_\phi P_{H^2_2} M_{\overline{\abf{1}_\alpha}}.\end{equation}
\end{thm}

\begin{rem*}
    Throughout, the results for $U^2_\alpha$ are simply the symmetric analogs (exchanging indices 1 and 2) of those for $U_\alpha^1$.
\end{rem*}

\begin{proof} Let $U_\alpha^1 = (V_\alpha^1)^*$,  where $V_\alpha^1$ was defined in Proposition \ref{prop:V1alpha}.  By Proposition \ref{prop:V1alpha} and the assumption that $J_\alpha$ is a surjective isometry (and hence unitary), we have $M_{z_1}  = J_\alpha U^1_\alpha J^*_\alpha$ or equivalently, $U_\alpha^1 = J_\alpha^* M_{z_1} J_\alpha$. Since $M_{z_1}$ is unitary on $L^2(\sigma_\alpha)$, the operator $U_\alpha^1$ is unitary on $K_\phi$. Furthermore, since $\abf{1}_\alpha$ is bounded, multiplication by $\abf{1}_\alpha$ and by $\overline{\abf{1}_\alpha}$  (denoted $M_{\abf{1}_\alpha}$ and $M_{\overline{\abf{1}_\alpha}}$) are both bounded operators on $L^2(\mathbb{T}^2)$. 

 To find the formula for $U_\alpha^1$, fix $f, g\in K_\phi$. Then
\[ 
\begin{aligned}
\langle U_\alpha^1 f, g \rangle_{K_\phi} &=  \langle  f, V_\alpha^1g \rangle_{K_\phi} \\
&=  \langle  f, B_1 g +\abf{1}_\alpha P_{H^2_2} g \rangle_{K_\phi} \\
& = \langle  S_\phi^1 f,  g \rangle_{K_\phi} +  \langle  M_{\overline{\abf{1}_\alpha}} f,  P_{H^2_2} g \rangle_{L^2} \\
& =  \langle  S_\phi^1 f,  g \rangle_{K_\phi} +  \langle   P_{H^2_2} M_{\overline{\abf{1}_\alpha}} f,  g \rangle_{H^2} \\
& =  \langle ( S_\phi^1 + P_{\phi} P_{H^2_2} M_{\overline{\abf{1}_\alpha}} ) f,  g \rangle_{K_\phi}, 
\end{aligned}
\]
which gives the desired formula for $U_\alpha^1.$
\end{proof}

We have the following two corollaries.

\begin{cor} \label{cor:cross1} Let $J_\alpha: K_\phi \rightarrow L^2(\sigma_\alpha)$ be surjective and $\abf{1}_\alpha$ bounded on $\mathbb{D}^2$. If $\phi(0,z_2) \equiv 0$, then $U^1_\alpha = J_\alpha^* M_{z_1} J_\alpha$, where 
\[ (U^1_\alpha f)(z)  = (S_\phi^1 f)(z) + \left \langle f, \abf{1}_\alpha k_{(z_2,0)} \right \rangle_{H^2}, \text{ for all } f \in K_\phi, \  z \in \mathbb{D}^2.\]
\end{cor}

\begin{proof} If $\phi(0,z_2) \equiv 0$, then $H_2^2 \subseteq K_\phi$ and so, we can simplify  $U_\alpha^1$ to
\[ U_\alpha^1 = S_\phi^1  +  P_{H^2_2} M_{\overline{\abf{1}_\alpha}}.\]
 To see the desired formula, fix $f \in K_\phi$ and $z_2 \in \mathbb{D}$. Then
\[ 
\begin{aligned}
( P_{H^2_2} M_{\overline{\abf{1}_\alpha}} f)( z_2) &= \langle P_{H^2_2} M_{\overline{\abf{1}_\alpha}} f, k_{(0, z_2)} \rangle_{H^2} \\
& = \langle M_{\overline{\abf{1}_\alpha}} f, k_{(0, z_2)} \rangle_{L^2}  \\
& = \langle f, \abf{1}_\alpha k_{(0, z_2)} \rangle_{H^2}, 
\end{aligned}
\]
which completes the proof.
\end{proof}

We can rewrite the formula in Corollary \ref{cor:cross1} in the following particularly clear way:

\begin{cor} \label{cor:crossformula} Let $J_\alpha: K_\phi \rightarrow L^2(\sigma_\alpha)$ be surjective and $\abf{1}_\alpha$ bounded on $\mathbb{D}^2$. Assume $\phi(0,z_2) \equiv 0$, so $\phi = z_1 \psi$ for some inner function $\psi$. Then 
\[ K_\phi = K_\psi \oplus \psi H^2_2\]
and for each $f \in K_\phi$, we can write $f = f_1 + \psi f_2$ for $f_1 \in K_\psi$ and $f_2 \in H^2_2.$ Then $U_\alpha^1$ is given by
\[ U_\alpha^1 f = z_1 f_1 + \alpha f_2.\]
\end{cor}

\begin{proof} By looking at the reproducing kernels of the underlying spaces and the fact that the spaces in each decomposition are orthogonal, we can conclude that 
\begin{equation} \label{eqn:decompcross} K_\phi = K_\psi \oplus \psi H^2_2 = H^2_2 \oplus z_1 K_\psi. \end{equation}
Furthermore, the proof of Corollary \ref{cor:cross1} gives 
\[ U_\alpha^1 = S_\phi^1  +  P_{H^2_2} M_{\overline{\abf{1}_\alpha}}.\]
Now fix $f \in K_\phi$ and write $f = f_1 + \psi f_2$ for $f_1 \in K_\psi$ and $f_2 \in H^2_2.$
We first compute $S_\phi^1 f$. To that end, note that \eqref{eqn:decompcross} gives $z_1 f_1 \in z_1 K_\psi \subseteq K_\phi.$ Thus, $S_\phi^1 f_1 = z_1 f_1$. Similarly, $S_\phi^1 (\psi f_2) = P_\phi (z_1 \psi f_2) =0$ since $z_1 \psi f_2 \perp H^2_2$ and $z_1 \psi f_2 \perp z_1 K_\psi$. Combining those formulas gives 
\[S_\phi^1 f = S_\phi^1 (f_1 + \psi f_2)=  z_1 f_1. \]
Now we compute $P_{H^2_2} M_{\overline{\abf{1}_\alpha}} f$. Again we write $f = f_1 + \psi f_2$ for $f_1 \in K_\psi$ and $f_2 \in H^2_2.$ Note that from \eqref{eqn:psi_alpha}, $\abf{1}_{\alpha}(z) = \overline{\alpha}\psi$. Let $h\in H_2^2$ and compute
\begin{align*}
    \left \langle P_{H^2_2} M_{\overline{\abf{1}_\alpha}} f_1, h \right \rangle_{H^2_2} &= \left \langle  M_{\overline{\abf{1}_\alpha}} f_1, h \right \rangle_{L^2} \\
    &= \left \langle  \overline{\abf{1}_\alpha} f_1, h \right \rangle_{L^2} \\
    &= \left \langle  \alpha \, \overline{\psi} \, f_1, h \right \rangle_{L^2} \\
    & =  \alpha \left \langle\, f_1,  \psi h \right \rangle_{L^2}.
\end{align*}

Since $f_1 \in K_\psi$ and $\psi h \in \psi H^2_2$, we get $P_{H^2_2} M_{\overline{\abf{1}_\alpha}} f_1 = 0$. Next, we compute $P_{H^2_2} M_{\overline{\abf{1}_\alpha}} (\psi f_2)$. For that note $P_{H^2_2} M_{\overline{\abf{1}_\alpha}} (\psi f_2) = P_{H^2_2} \overline{\abf{1}_\alpha} \psi f_2 = P_{H^2_2} (\alpha \, \overline{\psi} \, \psi f_2) = P_{H^2_2} (\alpha f_2) = \alpha f_2$ as $f_2 \in H^2_2.$ Thus $P_{H^2_2} M_{\overline{\abf{1}_\alpha}} f = \alpha f_2$.  Combining the above two computations, we get the desired formula for $U_\alpha^1$.
\end{proof}

\begin{rem*} In Corollary \ref{cor:crossformula}, the range of $P_{H^2_2} M_{\overline{\abf{1}_\alpha}}$ is all of $H^2_2$. Thus, in this case, it is immediate that $U_\alpha^1$ gives us an infinite-rank unitary perturbation of the compressed shift $S_\phi^1$.
\end{rem*}

Recall the definition of generic and exceptional $\alpha $ from Subsection \ref{subsec:RIFClark}.
\begin{cor} \label{cor:RIFUalpha} Let $\phi$ be a rational inner function on $\mathbb{D}^2$ with $\deg_j \phi  \ge 1$ for $j=1,2$ and let $\alpha \in \mathbb{T}$ be generic for $\phi$. Then $J_\alpha$ is surjective, $\abf{1}_\alpha$ is bounded, and $U^1_\alpha = J_\alpha^* M_{z_1} J_\alpha$, where 
\[ U^1_\alpha  = S_\phi^1  + P_\phi P_{H^2_2} M_{\overline{\abf{1}_\alpha}}.\]
\end{cor}

\begin{proof} Since $\alpha$ is generic for $\phi$,  Theorem \ref{theorem:poltoratski}  implies that $J_\alpha$ is unitary. Since $\phi$ is bounded on $\mathbb{D}^2$, so is $B_1 \phi$. If the function
\[ \frac{1}{1- \bar{\alpha} \phi(0, z_2)}\]
were not bounded on $\mathbb{D}$, then there would exist a $\tau \in \mathbb{T}$ and a sequence $(z_n) \subseteq \mathbb{D}$ converging to $\tau$ such that $\phi(0, z_n)$ converged to $\alpha.$ Since $\phi$ does not have any singularities on $\mathbb{D} \times \mathbb{T}$, we can conclude that  $\phi(0, \tau) = \alpha.$ Consider the finite Blaschke product
$\phi_\tau(z_1) := \phi(z_1, \tau)$.  Since $\phi_\tau(0) = \alpha$, this means $\phi_\tau$ must be identically $\alpha$. Thus, the set $\mathbb{T} \times \{\tau\}$ is contained in the level set $\mathcal{C}_\alpha$. This implies that $\alpha$ is not generic for $\phi$, a contradiction. Thus, $\abf{1}_\alpha$ must be bounded on $\mathbb{D}^2$ and the rest of the statement follows from Theorem \ref{thm:U1alpha}.
\end{proof}

Note that in Corollary \ref{cor:crossformula}, the projection $P_\phi$ is unnecessary in the second term in $U_\alpha^1$. In contrast, the following proposition shows that $P_\phi$ is often required.

\begin{prop} \label{prop:ptheta} Let $J_\alpha: K_\phi \rightarrow L^2(\sigma_\alpha)$ be surjective and let $\abf{1}_\alpha$ be bounded on $\mathbb{D}^2$, with $\frac{1}{1- \bar{\alpha} \phi(0, z_2)} \in H^2_2.$ Assume that $\phi(0, z_2) \not \equiv 0$ and $\phi$ depends on both $z_1$ and $z_2$. Then $P_\phi$  is needed in the $U_\alpha^1$ formula from \eqref{eqn:Ualpha1}. 
\end{prop}

\begin{proof} As was explained in Subsection \ref{aglerdec} (see also \cite{BickelLiaw}, for example), there exist positive kernels $K_1, K_2\colon \mathbb{D}^2 \times \mathbb{D}^2 \rightarrow \mathbb{C}$ such that for all $z, w\in \mathbb{D}^2,$
\begin{equation} \label{eqn:agler} k^\phi_w(z) =  \frac{1-\phi(z) \overline{\phi(w)}}{(1-z_1 \bar{w}_1)(1-z_2 \bar{w}_2)} = \frac{K_1(z,w)}{1-z_1 \bar{w}_1}+ \frac{K_2(z,w)}{1-z_2 \bar{w}_2}.\end{equation}
Then, for each $w\in \mathbb{D}^2$, $g_w(z) := \frac{ \bar{w}_1 K_2(z,w)}{1-z_2 \bar{w}_2} \in K_\phi$ and after multiplying through by $1-z_1\bar{w}_1$, we can rewrite \eqref{eqn:agler} as
\begin{equation} \label{eqn:agler2}  (T_1 g_w)(z) = z_1 g_w(z) = \frac{K_2(z,w)}{1-z_2 \bar{w}_2} + K_1(z,w)+ \phi(z) \overline{\phi(w)}k_{(0,w_2)}(z) - k_{(0,w_2)}(z). \end{equation}
Because $\phi$ depends on both $z_1, z_2$ and does not vanish identically when $z_1 =0$, we can choose $w = (w_1, w_2) \in \mathbb{D}^2$ so that 
\[ \phi(0, w_2) \ne 0 \ \ \text{ and } \ \ \phi( w_1, w_2) \ne \phi(0, w_2).\]
Since $\phi(0, w_2) \ne 0$, $\langle \phi, k_{(0, w_2)} \rangle_{H^2} \ne 0$ and so $k_{(0, w_2)} \not \in K_\phi$. To show that $P_\phi$ is needed in $U_\alpha^1$, we will show that there is a constant $c \ne 0$ so that $P_{H^2_2} M_{\overline{\abf{1}_\alpha}} g_w = c k_{(0, w_2)}$.

To that end, fix $\zeta \in \mathbb{D}$ and observe that 
\[ 
\begin{aligned}
\left( P_{H^2_2} M_{\overline{\abf{1}_\alpha}} g_w \right) (0, \zeta)  &= \left \langle P_{H^2_2} M_{\overline{\abf{1}_\alpha}} g_w, k_{(0, \zeta)} \right \rangle_{H^2_2} \\
& = \left \langle M_{\overline{\abf{1}_\alpha}} g_w, k_{(0, \zeta)} \right \rangle_{L^2} \\
& = \left \langle g_w, \bar{\alpha} \frac{B_1\phi}{1- \bar{\alpha}\phi(0, \cdot)} k_{(0, \zeta)} \right \rangle_{H^2} \\
 & = \alpha \left \langle g_w, T_1^*\left(\frac{\phi}{1- \bar{\alpha}\phi(0, \cdot)} k_{(0, \zeta)} \right) \right \rangle_{H^2}\\
& = \alpha \left \langle T_1 g_w, \frac{\phi}{1- \bar{\alpha}\phi(0, \cdot)} k_{(0, \zeta)} \right \rangle_{H^2}.\end{aligned}\]

Using \eqref{eqn:agler2} and the fact that $\frac{\phi}{1- \bar{\alpha}\phi(0, \cdot)} k_{(0, \zeta)} \in \phi H^2$ and $\frac{K_2(z,w)}{1-z_2 \bar{w}_2}, K_1(z,w) \in K_\phi$, we can conclude that
\[ 
\begin{aligned}
&\left( P_{H^2_2} M_{\overline{\abf{1}_\alpha}} g_w \right) (0, \zeta) 
= \alpha \left \langle \phi \overline{\phi(w)}k_{(0,w_2)} - k_{(0,w_2)}, \frac{\phi}{1- \bar{\alpha}\phi(0, \cdot)} k_{(0, \zeta)} \right \rangle_{H^2} \\
& =  \alpha \overline{\phi(w)} \left \langle k_{(0,w_2)}, \frac{1}{1- \bar{\alpha}\phi(0, \cdot)} k_{(0, \zeta)} \right \rangle_{H^2} -\alpha \left \langle k_{(0,w_2)}, \frac{\phi}{1- \bar{\alpha}\phi(0, \cdot)} k_{(0, \zeta)} \right \rangle_{H^2}\\
& = \alpha \left( \frac{\overline{\phi(w_1, w_2)} - \overline{\phi(0,w_2)}}{1- \alpha\overline{\phi(0, w_2)}} \right) k_{(0, w_2)}(0, \zeta) \\
& = c k_{(0, w_2)}(0, \zeta), 
\end{aligned}\]
for some $c \ne 0.$ Since $k_{(0, w_2)} \not \in K_\phi,$  the projection  $P_\phi$ is required to ensure that the operator $U_\alpha^1$ maps into $K_\phi$.
\end{proof}

\begin{rem} \label{rem:RIFthms} It is worth noting that if $\phi$ is a rational inner function that both depends on $z_1$ and $z_2$ and does not have a factor of $z_1$, then we can apply Proposition \ref{prop:ptheta} (and Theorem \ref{thm:perturbissmall} below)  to $\phi$ for all of its generic $\alpha \in \mathbb{T}.$ In particular, as discussed in Corollary \ref{cor:RIFUalpha} and its proof, $J_\alpha$ is surjective and $\frac{1}{1- \bar{\alpha} \phi(0, z_2)}$ is bounded and hence in $H^2_2$. For more details about an example of a $U_\alpha^1$ requiring $P_\phi$, see Subsection \ref{ex:Bp}.
 \end{rem}

While the perturbations needed to complete a compressed shift $S_1^{\phi}$ to a unitary are often infinite-rank operators, they are not too large in the following sense.

\begin{thm}\label{thm:perturbissmall}
Let $J_\alpha: K_\phi \rightarrow L^2(\sigma_\alpha)$ be surjective and let $\abf{1}_\alpha$ be bounded on $\mathbb{D}^2$, with $\frac{1}{1- \bar{\alpha} \phi(0, z_2)} \in H^2_2.$ As in \eqref{maxmindecomp}, write $K_{\phi}=S_1^{\mathrm{max}}\oplus S^{\mathrm{min}}_2$. Then
\[P_{\phi}P_{H^2_2}M_{\overline{\vartheta_{\alpha}}}\equiv0 \quad \textrm{on}\quad S_1^{\mathrm{max}}.\]    
\end{thm}

\begin{proof} Fix $f \in S_1^{\mathrm{max}}$ and $w_2 \in \mathbb{D}.$ By our assumptions 
\[ g_{w_2}(z_2):=\frac{1}{1- \bar{\alpha} \phi(0, z_2)}\frac{1}{1- \bar{w}_2z_2} \in H^2_2.\]
Proposition 3.5 in \cite{bk12}, which says $(B_1 \phi) H^2_2 \subseteq S^{\mathrm{min}}_2$, implies that 
\[ h_{w_2}(z):=\frac{\abf{1}_\alpha(z)}{1- \bar{w}_2z_2} = \bar{\alpha} (B_1 \phi)(z)g_{w_2}(z_2) \in S^{\mathrm{min}}_2. \]
From this, we can conclude that 
\[
\begin{aligned}
\left(P_{H^2_2}M_{\overline{\vartheta^1_{\alpha}}}f \right)(0, w_2)& =  
\left \langle P_{H^2_2} M_{\overline{\abf{1}_\alpha}} f, k_{(0, w_2)} \right \rangle_{H^2_2} \\
& =  
\left \langle M_{\overline{\abf{1}_\alpha}} f, k_{(0, w_2)} \right \rangle_{L^2} \\
& = \left \langle  f, h_{w_2} \right \rangle_{H^2} =0,
\end{aligned}
\]
since $S_1^{\mathrm{max}} \perp_{H^2} S^{\mathrm{min}}_2.$ This immediately implies that 
\[ P_{\phi}P_{H^2_2}M_{\overline{\vartheta_{\alpha}}} f =P_{\phi} 0 =0, \]
which completes the proof.    
\end{proof}

\section{RIF Level Sets and Taylor Joint Spectrum}\label{sec:spec}

We are interested in the Taylor joint spectrum of the shift operators $(M_{z_1}, M_{z_2})$  on a given space of functions $L^2(\mu)$. 

To that end, we need to establish some notation. First, let $\mathrm{supp}\,\mu$ denote the support of $\mu$. 
Now, fix $(\lambda_1, \lambda_2)  \in \mathbb{T}^2$ and write $(\lambda_1, \lambda_2) = (e^{i \tau_1}, e^{i \tau_2})$ with $-\pi < \tau_j \le \pi$ for $j=1,2$. For $\epsilon >0$ and $j=1,2$, define the sets
\[ 
\begin{aligned}
V^j_\epsilon&:= V^j_\epsilon(\lambda_1, \lambda_2) = \left \{ (e^{i \theta_1}, e^{i \theta_2}) : \tau_j - \tfrac{\epsilon}{2} < \theta_j < \tau_j + \tfrac{\epsilon}{2}  \right \}, \\
V_\epsilon &:= V_\epsilon(\lambda_1, \lambda_2)  = V^1_\epsilon \cap V^2_\epsilon.
\end{aligned}
 \]
For simplicity, we assume that $\lambda_1, \lambda_2 \ne -1$ and $\epsilon$ is chosen sufficiently small so that $[\tau_j-\tfrac{\epsilon}{2} , \tau_j+\tfrac{\epsilon}{2}] \subseteq (-\pi, \pi]$ for $j=1,2.$ Slight modifications handle the case when $\lambda_j = -1.$

Then we have the following theorem, whose content is certainly well-known to experts. 

\begin{thm} \label{thm:jointspectrum}
Let $\mu$ be a finite positive Borel measure on $\mathbb{T}^2$ and assume that $\mu$ behaves like Lebesgue measure in the following sense:
there is a dense subset $\mathcal{S}$ of $\mathrm{supp}\,\mu$ such that for each $(\lambda_1, \lambda_2) \in \mathcal{S}$, there is 
some $c >0$ such that for $\epsilon >0$ sufficiently small, $\mu ( V_\epsilon(\lambda_1, \lambda_2) ) \ge c \epsilon^2.$ 

Let $M_{z_1}$, $M_{z_2}$ denote multiplication by $z_1, z_2$ on $L^2(\mu)$. Then  $\tay(M_{z_1}, M_{z_2}) =  \mathrm{supp}\,\mu.$
\end{thm}

\begin{rem}\label{r-Vasi}
It was pointed out 
to us by Michael Hartz that the same conclusion, namely that $\tay(M_{z_1}, M_{z_2})=\mathrm{supp}\,\mu$, holds without the additional hypothesis on $\mu$, as can be seen by appealing for instance to a result of Vasilescu (see \cite[Theorem 1.1]{vas}). We have opted to present a direct proof so as to keep our paper self-contained and highlight some specific properties of Clark measures associated to rational inner functions. 
\end{rem}

\begin{proof} For the first part of the proof, we will establish that $\tay(M_{z_1}, M_{z_2}) \subseteq  \mathrm{supp}\,\mu$ by showing
if $(\lambda_1, \lambda_2) \not \in \mathrm{supp}\, \mu$,
 then $(\lambda_1, \lambda_2) \not \in \tay(M_{z_1}, M_{z_2})$. This means that we need to show that $ \text{Ker}( \delta_n)  \subseteq \text{Ran} (\delta_{n-1})$ for $n=1, 2, 3$. (For the definitions of $\delta_1, \delta_2, \delta_3,$ see Subsection \ref{subsec:Tspec}.) 
 
For $n=1$, note that if $h \in  \text{Ker}( \delta_1)$, then 
\[ h(z_1, z_2) (z_1 - \lambda_1) = h(z_1,z_2)(z_2-\lambda_2) =0\]
for $\mu$ a.e. $z$. Then $h(z_1, z_2) =0$ for $\mu$-a.e. $(z_1, z_2) \in \mathbb{T}^2 \setminus \{(\lambda_1, \lambda_2)\}$. 
Since  $(\lambda_1, \lambda_2) \not \in \mathrm{supp}\,\mu$, 
this implies that $h\equiv 0$ in $L^2(\mu)$, so $h \in \text{Ran} (\delta_{0}) =\{0\}.$

For $n=2$, assume that $(f_1, f_2) \in \text{Ker}( \delta_2)$. Then
\begin{equation} \label{eqn:kerdelta2} (z_1-\lambda_1) f_2(z_1, z_2) =  (z_2-\lambda_2) f_1(z_1, z_2)\end{equation}
for $\mu$ a.e. $(z_1, z_2) \in \mathbb{T}^2$.  Define 
\[ h(z_1, z_2) = \left\{ \begin{array}{ll} \tfrac{f_1(z_1, z_2)}{z_1 - \lambda_1} & \text{ for } (z_1, z_2) \in \mathbb{T}^2 \setminus ( \{\lambda_1\} \times \mathbb{T}) \\ 
\tfrac{f_2(z_1, z_2)}{z_2 - \lambda_2} & \text{ for } (z_1, z_2) \in  \mathbb{T}^2 \setminus ( \mathbb{T} \times \{\lambda_2\} ).  
\end{array} \right.\]
 By assumption \eqref{eqn:kerdelta2} and since $(\lambda_1, \lambda_2) \not \in \mathrm{supp}\,\mu,$ $h$ is well-defined $\mu$-a.e. on $\mathbb{T}^2$. Furthermore, note that for $\mu$-a.e. $(z_1, z_2) \in \mathbb{T}^2 \setminus ( \{\lambda_1\} \times \mathbb{T})$, the definition of $h$ implies 
\[ (M_{z_1} - \lambda_1) h(z_1, z_2) = f_1(z_1, z_2).\]
Meanwhile, for $\mu$-a.e. $(z_1, z_2) \in \mathbb{T}^2 \setminus (\mathbb{T} \times \{\lambda_2\})$, the definition of $h$ and \eqref{eqn:kerdelta2} implies that
\[ (M_{z_1} - \lambda_1) h(z_1, z_2) = \left(\frac{z_1-\lambda_1}{z_2-\lambda_2} \right) f_2(z_1, z_2) = f_1(z_1, z_2).\]
Similar arguments hold for $(M_{z_2} -\lambda_2)h$ and so
\begin{equation} \label{eqn:d1h} \left( (M_{z_1} - \lambda_1) h, (M_{z_2}-\lambda_2)h \right) = \left( f_1, f_2 \right)\end{equation}
$\mu$-a.e. Now, we just need to show that $h \in L^2(\mu)$. Once that is established, \eqref{eqn:d1h} will give
$\delta_1 (h) = \left( f_1, f_2 \right)$, so $(f_1, f_2) \in \text{Range}(\delta_1)$, as needed.

To show $h \in L^2(\mu)$,  note that since $(\lambda_1, \lambda_2) \not \in \mathrm{supp}\,\mu$, there is an $\epsilon>0$ so that $\mu(V_\epsilon(\lambda_1, \lambda_2) )=0$. To show $h \in L^2(\mu)$, define
\[ 
\begin{aligned}
S_1 &= \mathbb{T}^2 \setminus  V^1_\epsilon(\lambda_1, \lambda_2) \\
S_2 & = V^1_\epsilon(\lambda_1, \lambda_2)  \setminus V_\epsilon(\lambda_1, \lambda_2),   \end{aligned}
\]
so $\mathbb{T}^2 = S_1 \cup S_2 \cup V_\epsilon(\lambda_1, \lambda_2)$. 
Then choosing $\epsilon$ sufficiently small gives
\[ 
\begin{aligned} 
\| h\|^2_{L^2(\mu)} &= \int_{S_1} |h(z)|^2  d \mu +  \int_{S_2} |h(z)|^2 d \mu  + \int_{V_\epsilon(\lambda_1, \lambda_2)} |h(z)|^2 d \mu  \\
&= \int_{S_1} \left| \tfrac{f_1(z_1, z_2)}{z_1 - \lambda_1}  \right|^2 d \mu + \int_{S_2} \left| \tfrac{f_2(z_1, z_2)}{z_2 - \lambda_2}  \right|^2 d \mu + 0 \\
&\lesssim \frac{1}{\epsilon^2} \| f_1\|^2_{L^2(\mu)} + \frac{1}{\epsilon^2} \| f_2\|^2_{L^2(\mu)} 
< \infty,
\end{aligned}
\]
as needed. 

For the $n=3$ case, note that $\text{Ker}( \delta_3) = L^2(\mu)$. So, we need to fix $h \in L^2(\mu)$ and show that $h \in \text{Ran} (\delta_{2})$. This means that we need $f_1, f_2 \in L^2(\mu)$ such that
\begin{equation} \label{eqn:hdecomp} h(z_1, z_2) = (z_1 - \lambda_1)f_2(z_1, z_2) -  (z_2 -\lambda_2) f_1(z_1, z_2),\end{equation}
for $\mu$-a.e. $(z_1, z_2) \in \mathbb{T}^2$. Let 
\[ f_2 = \frac{h}{z_1 -\lambda_1} 1_{S_1} \text{ and  } f_1 = \frac{-h}{z_2 -\lambda_2} 1_{S_2},\] 
where $S_1$ and $S_2$ were defined in the $n=2$ step. By the definitions of $S_1$ and $S_2$ and the fact that $\mu(V_\epsilon(\lambda_1, \lambda_2)) =0$, we can conclude that \eqref{eqn:hdecomp} holds and $f_1, f_2 \in L^2(\mu)$. This completes the proof that if $(\lambda_1, \lambda_2) \not \in \mathrm{supp} \,\mu$, then  $(\lambda_1, \lambda_2) \not \in \tay (M_{z_1}, M_{z_2} )$.

For the other direction, first fix  $(\lambda_1, \lambda_2) \in \mathcal{S}$. Since $\mu$ is finite, $1 \in L^2(\mu)$. We will show that  $1 \not \in \text{Ran}(\delta_2)$ and hence, $L^2(\mu) = \text{Ker}(\delta_3) \ne \text{Ran}(\delta_2).$
By way of contradiction, assume that there exist $h_1, h_2 \in L^2(\mu)$ such that
\[ 1 = (z_1 - \lambda_1)h_2(z_1, z_2)  +  (z_2 - \lambda_2)h_1(z_1, z_2).\] 
for $\mu$ a.e. $(z_1, z_2) \in \mathbb{T}^2$.  First, if $\mu$
 has a point-mass at $(\lambda_1, \lambda_2)$ with weight $c \ne 0$, then $h_1, h_2$ would have to be well-defined at $(\lambda_1, \lambda_2)$ and so
 \[ c= \mu(\{(\lambda_1, \lambda_2)\}) = c (\lambda_1-\lambda_1) h_2(\lambda_1, \lambda_2) + c (\lambda_2-\lambda_2) h_1(\lambda_1, \lambda_2)=0, \]
a contradiction. 

Now assume that $\mu$ does not have a point-mass at $(\lambda_1, \lambda_2).$  Then for $\epsilon >0$ sufficiently small, 
\[ 
\begin{aligned}
\mu(V_\epsilon (\lambda_1, \lambda_2) ) &= \int_{V_\epsilon (\lambda_1, \lambda_2) }  |(z_1 - \lambda_1)h_2(z_1, z_2)  +  (z_2 - \lambda_2)h_1(z_1, z_2)|^2 d \mu \\
& \lesssim \epsilon^2   \int_{V_\epsilon (\lambda_1, \lambda_2) }  |h_1(z_1, z_2)|^2 + |h_2(z_1, z_2)|^2 d \mu \\
& = \epsilon^2  \left( \| h_1 1_{V_\epsilon (\lambda_1, \lambda_2) } \|^2_{L^2(\mu)} +  \| h_2 1_{V_\epsilon (\lambda_1, \lambda_2) }\|^2_{L^2(\mu)} \right),
\end{aligned}
 \]
 which combined with assumption (b) about $\mu$ implies that there is a $c>0$ with 
\[ c \epsilon^2 \lesssim \epsilon^2  \left( \| h_1 1_{V_\epsilon (\lambda_1, \lambda_2) } \|^2_{L^2(\mu)} +  \| h_2 1_{V_\epsilon (\lambda_1, \lambda_2) }\|^2_{L^2(\mu)}\right)\]
for $\epsilon$ sufficiently small.
Canceling the $\epsilon$ and using the dominated convergence theorem, this implies 
\[ c \lesssim \lim_{\epsilon \rightarrow 0}   \left( \| h_1 1_{V_\epsilon (\lambda_1, \lambda_2) } \|^2_{L^2(\mu)} +  \| h_2 1_{V_\epsilon (\lambda_1, \lambda_2) }\|^2_{L^2(\mu)}\right) =0,\]
since the $\lim_{\epsilon \rightarrow 0} |h_1(z) 1_{V_\epsilon (\lambda_1, \lambda_2) }(z)|^2 = 0 =  \lim_{\epsilon \rightarrow 0}|h_2(z) 1_{V_\epsilon (\lambda_1, \lambda_2) }(z)|^2$ for $\mu$-a.e. $z\in \mathbb{T}^2$ (where we used the fact that $\mu$ does not have a point mass at $(\lambda_1, \lambda_2)$).

This gives the contradiction and implies that $(\lambda_1, \lambda_2) \in \tay(M_{z_1}, M_{z_2})$. This implies that $ \tay(M_{z_1}, M_{z_2})$ contains a dense set in $\mathrm{supp}\,\mu$. Since the Taylor spectrum $ \tay(M_{z_1}, M_{z_2})$ is closed in $\mathbb{C}$, this implies that $\mathrm{supp}\,\mu \subseteq \tay(M_{z_1}, M_{z_2})$ and completes the proof.
\end{proof} 

To apply this to Clark measures $\sigma_{\alpha}$ associated to rational inner functions, we need the following lemma. Recall the definition of the level set of $\phi$, see Equation \eqref{e-Calpha}.

\begin{lm} \label{lem:genericepsilon} Let $\phi=\frac{\tilde{p}}{p}$ be a bidegree $(n_1,n_2)$ RIF with $n_1,n_2\geq 1$ and let $\alpha \in \T$ be generic for $\phi$. Fix $(\lambda_1, \lambda_2)$ in $\mathcal{C}_\alpha$ such that $(\lambda_1, \lambda_2)$ is not a singularity of $\phi$. Then, there is some $c >0$ such that for $\epsilon >0$ sufficiently small, 
\[ \sigma_\alpha ( V_\epsilon (\lambda_1, \lambda_2) ) \ge c \epsilon.\]
\end{lm}

\begin{proof} Our goal is to use the formula for $\sigma_\alpha$ from Theorem \ref{theorem:RIFclarkmeas}.  

We first need some preliminary facts. Note that since $\alpha$ is generic and $\phi(\lambda_1, \lambda_2) = \alpha$, then the one-variable functions $\phi(\cdot, \lambda_2)$ and $ \phi(\lambda_1, \cdot )$ cannot be constant. Since non-constant finite Blaschke products have non-zero derivatives on $\mathbb{T}$, we have $\frac{\partial \phi}{ \partial z_j}(\lambda_1, \lambda_2) \ne 0$ for $j=1,2$. By continuity, for $\epsilon >0$ sufficiently small, there is an $m>0$ such that $|\frac{\partial \phi}{\partial z_2}(z_1,z_2)| <m$ for all $(z_1, z_2) \in V_\epsilon(\lambda_1, \lambda_2)$.

As described in Subsection \ref{subsec:RIFClark}, $\mathcal{C}_\alpha$  is globally parameterized by a finite family of analytic functions on $\mathbb{T}^2$. However, the implicit function theorem implies that  $\mathcal{C}_\alpha$ is actually parameterized by a single analytic function $z_2 = g_\alpha(z_1)$ in a neighborhood of $(\lambda_1, \lambda_2)$.  It is also worth noting that $g_\alpha'(\lambda_1) \ne 0$. To see this, observe that 
\[ \phi(z_1, g_\alpha(z_1)) = \alpha \text{ implies that } \tfrac{\partial \phi}{\partial z_1}(\lambda_1, \lambda_2) +  \tfrac{\partial \phi}{\partial z_2}(\lambda_1, \lambda_2) g'_\alpha(\lambda_1) =0.\]
Then since $\frac{\partial \phi}{\partial z_1}(\lambda_1, \lambda_2) \ne 0,$  $a:=g_\alpha'(\lambda_1) \ne 0$. Thus, near $\lambda_1,$
\begin{equation} \label{eqn:galpha} g_\alpha(z_1) = \lambda_2 + a (\lambda_1 - z_1) + O(|\lambda_1-z_1|^2).\end{equation}

Theorem \ref{theorem:RIFclarkmeas}  is technically stated only for continuous functions. So, we will apply Theorem \ref{theorem:RIFclarkmeas} to a continuous function $f_\epsilon$ that is always less than $1_{V_{\epsilon}(\lambda_1, \lambda_2)}$ and equals $1$ on $V_{\frac{\epsilon}{2}}(\lambda_1, \lambda_2)$.  For $\epsilon>0$ sufficiently small this gives
\[
\begin{aligned}
 \sigma_\alpha(V_{\epsilon}(\lambda_1, \lambda_2)) &\ge \int_{\mathbb{T}^2} f_\epsilon(z_1, z_2)  \ d \sigma_\alpha(z_1, z_2) \\
 &\ge \int_{\mathbb{T}}  1_{V_{\frac{\epsilon}{2}}(\lambda_1, \lambda_2)} (\zeta, g_\alpha(\zeta))  \frac{1}{|\frac{\partial \phi}{\partial z_2} (\zeta, g_\alpha(\zeta))|} dm(\zeta) \\
 & \ge \frac{1}{m}   \int_{\mathbb{T}}  1_{V_{\frac{\epsilon}{2}}(\lambda_1, \lambda_2)} (\zeta, g_\alpha(\zeta)) \ dm(\zeta) \\
 & = \frac{1}{2\pi m} \int_{-\pi}^{ \pi} 1_{V_{\frac{\epsilon}{2}}(\lambda_1, \lambda_2)} (e^{i\theta}, g_\alpha(e^{i\theta})) \ d \theta.
 \end{aligned}
 \]
 Observe that $1_{V_{\frac{\epsilon}{2}}(\lambda_1, \lambda_2)} (e^{i\theta}, g_\alpha(e^{i\theta})) =1$ if and only if both of the following are satisfied:
 \[
 \begin{aligned}
 \tau_1 - \tfrac{\epsilon}{4} &\le \theta \le  \tau_1+  \tfrac{\epsilon}{4}, \\
 \tau_2 - \tfrac{\epsilon}{4} &\le \text{Arg}(g_\alpha(e^{i\theta})) \le  \tau_2 +  \tfrac{\epsilon}{4}.  
 \end{aligned}
 \]
 Furthermore, note that by \eqref{eqn:galpha} and the typical relationship between a unimodular constant and its principal argument, we have 
 \[ |  \text{Arg}(g_\alpha(e^{i\theta})) - \tau_2| \approx |g_\alpha(e^{i\theta}) - \lambda_2| \approx |e^{i\theta} - \lambda_1 | \approx | \theta - \tau_1|.\]
 Then there exists a $b_1 >0$ so that 
 \[  | \theta - \tau_1| < b_1 \epsilon  \text{ implies that } |  \text{Arg}(g_\alpha(e^{i\theta})) - \tau_2| < \tfrac{\epsilon}{4}.\]
 Setting $b = \min\{ \frac{1}{4}, b_1\}$, we can see that if $|\tau_1-\theta| < b \epsilon$, then \[1_{V_{\frac{\epsilon}{2}}(\lambda_1, \lambda_2)} (e^{i\theta}, g_\alpha(e^{i\theta})) =1.\]
 This implies
\[
 \sigma_\alpha(V_{\epsilon}(\lambda_1, \lambda_2)) \ge  \frac{1}{2\pi m} \int_{\tau_1 - b \epsilon}^{\tau_1 + b \epsilon} 1 \ d\theta =  \frac{b \epsilon}{\pi m},\]
 which completes the proof.
 \end{proof}

This tells us about the joint spectrum of unitary perturbations of the compressed shifts $(S_\phi^1, S_\phi^2).$

\begin{thm}\label{thm:levsetsasTaylorspec}  Let $\phi=\frac{\tilde{p}}{p}$ be a bidegree $(n_1,n_2)$ RIF with $n_1,n_2\geq 1$ and let $\alpha \in \T$ be generic for $\phi$. Then there is a commuting pair of unitary perturbations $(U_\alpha^1, U_\alpha^2)$ of the compressed shifts $(S_\phi^1, S_\phi^2)$ on $K_\phi$ such that their Taylor joint spectrum $\tay(U_\alpha^1, U_\alpha^2) = \mathcal{C}_\alpha$.
\end{thm} 

\begin{proof} First, we show that on $L^2(\sigma_\alpha)$, we have $\tay(M_{z_1}, M_{z_2}) = \mathcal{C}_\alpha$. 
By Theorem \ref{theorem:RIFclarkmeas}, the set $\mathcal{C}_\alpha$ is the support of $\sigma_\alpha$. To apply Theorem \ref{thm:jointspectrum}, note that $\mathcal{C}_\alpha$ is the closure of the set 
\[ S_\alpha = \{ (\lambda_1, \lambda_2) \in \mathbb{T}^2: \phi \text{ is holomorphic at } (\lambda_1, \lambda_2) \text { and } \phi(\lambda_1, \lambda_2) =\alpha\}.\] 
By Lemma \ref{lem:genericepsilon}, if $(\lambda_1, \lambda_2) \in S_\alpha$, there is a $c>0$ such that for $\epsilon>0$ sufficiently small,
\[ \sigma_\alpha ( V_\epsilon (\lambda_1, \lambda_2) ) \ge c \epsilon \ge c \epsilon^2,\]
as needed. Then $\tay(M_{z_1}, M_{z_2}) = \mathcal{C}_\alpha$ follows immediately from Theorem \ref{thm:jointspectrum}.

To move this result to the model space $K_\phi$, note that Corollary \ref{cor:RIFUalpha} implies that $J_\alpha$ is unitary and there are unitary maps $U_\alpha^1, U_\alpha^2: K_\phi \rightarrow K_\phi$ such that 
$U^j_\alpha = J_\alpha^* M_{z_j} J_\alpha$ for $j=1,2$. Since $M_{z_1},  M_{z_2}$ commute on $L^2(\sigma_\alpha)$, the operators $U_\alpha^1, U_\alpha^2$ commute on $K_\phi$. 
Moreover, 
\[ 
\begin{aligned}
U^1_\alpha  &= S_\phi^1  + P_\phi P_{H^2_2} M_{\overline{\abf{1}_\alpha}} \\
U^2_\alpha  &= S_\phi^2  + P_\phi P_{H^2_1} M_{\overline{\abf{2}_\alpha}}, 
\end{aligned}
\]
where $\abf{2}_\alpha(z) = \bar{\alpha} \frac{(B_2\phi)(z)}{1- \bar{\alpha}\phi(z_1, 0)}$ and $B_2$ denotes the backward shift in $z_2$. (As an aside, while Corollary \ref{cor:RIFUalpha} only states the information about $U_\alpha^1$, the proof immediately implies the corresponding information about $U_\alpha^2$.) Thus, the operators $U^1_\alpha, U^2_\alpha$ are commuting unitary perturbations of the compressed shifts $S_\phi^1, S_\phi^2$. By above, $\tay(M_{z_1}, M_{z_2}) = \mathcal{C}_\alpha$. Since the Taylor joint spectrum is invariant under unitary conjugation, this gives $\tay(U_\alpha^1, U_\alpha^2) = \mathcal{C}_\alpha$.
\end{proof}

\begin{rem*}
There are several techniques for constructing rational inner functions $\phi$ whose level sets $\mathcal{C}_{\alpha}$ have prescribed properties, see for instance \cite[Section 7]{BPSII}. By Theorem \ref{thm:levsetsasTaylorspec}, this in turn allows us to prescribe such curve subsets of $\mathbb{T}^2$ as Taylor spectra of natural pairs of commuting unitaries on model spaces.
\end{rem*}
\section{Examples}\label{sec:examples}

\subsection{Finite Blaschke Products} \label{ex:Bp} In this subsection, we use inner functions constructed from finite Blaschke products to explore our various results.  Specifically, let $\phi_1$ and $\phi_2$ be finite Blaschke products (with at least one nonconstant) and consider $\phi(z) = \phi(z_1) \phi(z_2)$. 

\begin{example} First, let us  assume that one of $\phi_1$ or $\phi_2$ is constant, so that $\phi$ is a function of one variable. Then for each $\alpha \in \mathbb{T},$ the level set $\mathcal{C}_\alpha$ will contain lines of either the form $\mathbb{T} \times \{\tau\}$ or $\{\tau\} \times \mathbb{T}$ for some $\tau \in \mathbb{T}.$ Then basically by applying the argument in the proof of Proposition 3.10 in  \cite{BickelSolaI}, one can show that the isometry $J_\alpha$ linking $K_\phi$ and $L^2(\sigma_\alpha)$ is not unitary. Thus, our machinery breaks down and we cannot apply Theorem \ref{thm:U1alpha} to obtain a unitary perturbation of $S_\phi^1.$

Indeed, if $\phi_1$ is constant, then 
\[ K_\phi = \mathcal{H}\left( \frac{1-\phi_2(z_2)\overline{\phi_2(w_2)}}{(1-z_2 \bar{w}_2)(1-z_1 \bar{w}_1)} \right) = K^2_{\phi_2} \otimes H^2_1(\mathbb{D}), \]
where $K^2_{\phi_2}$ is the one-variable model space associated to $\phi_2$ with variable $z_2$. By this formula, it is immediate that $K_\phi$ is invariant under the $z_1$-shift $T_1$ and so, $S_\phi^1 = T_1 |_{K_{\phi}}$ is an isometry. Thus, in this case, $S_\phi^1$ strongly resembles the shift on $H^2(\mathbb{D})$. 
\end{example}

For this reason, it makes sense to restrict attention to $\phi_1, \phi_2$ that are both nonconstant. In this situation, since $\phi$ also has no boundary singularities, every $\alpha \in \mathbb{T}$ must be generic for $\phi$. From this,  we can apply Corollary \ref{cor:RIFUalpha} to obtain the formula 
\[U^1_\alpha  = S_\phi^1  + P_\phi P_{H^2_2} M_{\overline{\abf{1}_\alpha}} \]
for the associated perturbation of $S_\phi^1$ to the unitary $U_\alpha^1$.

\begin{example} If $\phi_1(0)=0$, then $\psi_1(z_1) := \phi_1(z_1)/z_1$ is also a finite Blaschke product and $\phi(z) = z_1 \psi_1(z_1) \phi_2(z_2)$ with
$\phi(0, z_2) \equiv 0$. In this case, we have a particularly nice formula for $U_\alpha^1.$ We can apply Corollary \ref{cor:crossformula} to conclude that 
\[ K_\phi = K_{\psi_1 \phi_2} \oplus \psi_1 \phi_2 H^2_2\]
and for each $f \in K_\phi$, we can write $f = f_1 + \psi_1 \phi_2 f_2$ for $f_1 \in K_{\psi_1 \phi_2} $ and $f_2 \in H^2_2.$ Then with respect to this decomposition, $U_\alpha^1$ is given by
\[ U_\alpha^1 f = z_1 f_1 + \alpha f_2.\]
Since $\| f \|_{K_\phi}^2 = \| f_1 \|_{H^2}^2 + \|f_2\|_{H^2}^2$, it is easy to see that $U_\alpha^1$ preserves norms. One can also see that it is surjective by considering the following alternative decomposition of $K_\phi:$ $K_\phi = H^2_2 \oplus z_1 K_{\psi_1 \phi_2}.$
\end{example}

Meanwhile, if $\phi_1(0) \ne 0$, the formula for $U_\alpha^1$ is more opaque. In the following example, we highlight some useful computations that shed some light on the formula.

\begin{example} Assume $\phi_1(0)\ne 0$. The product structure of $\phi$ still allows us to decompose $K_\phi$ as 
\[ K_\phi = K_{\phi_2} \oplus \phi_2 K_{\phi_1},\]
 where $K_{\phi_2}$ is invariant under $T_1,$ the shift in $z_1$ and thus $K_{\phi_2} \subseteq S_1^{\mathrm{max}}$. By Remark \ref{rem:RIFthms}, $\phi$ satisfies the assumptions of Theorem \ref{thm:perturbissmall} and so,
$(U_\alpha^1 - S_\phi^1)|_{K_{\phi_2}} \equiv 0.$ 
Thus, fully understanding $U_\alpha^1$ would require us to more fully understand 
\[(U_\alpha^1 - S_\phi^1)|_{\phi_2 K_{\phi_1}} =  (P_\phi P_{H^2_2} M_{\overline{\abf{1}_\alpha}})|_{\phi_2 K_{\phi_1}}.\]
Initially, one might hope that this formula reduces to a simple adaptation of the one-variable unitary perturbation of $S_{\phi_1}$ on the one variable space $K_{\phi_1}^1.$ Unfortunately, this does not appear to be the case. Indeed, if $h \in K^1_{\phi_1}$ and $ g \in H^2(\mathbb{D})$, then
\[f(z)= \phi_2(z_2) h(z_1) g(z_2) \in \phi_2 K_{\phi_1},\]
which gives a nice class of functions on which to test $(U_\alpha^1 - S_\phi^1)|_{\phi_2 K_{\phi_1}}$. Then we have
\[ \begin{aligned}(U_\alpha^1 - S_\phi^1) f & = P_\phi P_{H^2_2} \overline{\abf{1}_\alpha} f \\
& = P_\phi P_{H^2_2} \left( \alpha \frac{\overline{(B_1\phi)(z)}}{1- \alpha\overline{\phi(0, z_2)}} f(z) \right)\\
& = \alpha  P_\phi P_{H^2_2}  \left(\frac{\overline{\phi_2(z_2)(B_1\phi_1)(z_1)}}{1- \alpha\overline{\phi_1(0)\phi_2( z_2)}} \phi_2(z_2) h(z_1) g(z_2)\right)  \\
& = \alpha  P_\phi P_{H^2_2} \left(\frac{\overline{(B_1\phi_1)(z_1)}}{1- \alpha\overline{\phi_1(0)\phi_2( z_2)}} h(z_1) g(z_2) \right). 
\end{aligned}
\]
Because of the involvement of $P_{H^2_2}$  and $g(z_2)$, this formula does not appear to nicely mimic the one-variable situation. Still, for some special functions, we can compute parts of this formula more directly. For example, if we set $h = B_1 \phi_1$ and $g=1$ to get
\[ F(z) : =\phi_2(z_2) (B_1 \phi_1) (z_1) \in \phi_2 K_{\phi_1}, \]
then
\[   M_{\overline{\abf{1}_\alpha}} F := \alpha \frac{\overline{(B_1\phi)(z)}}{1- \alpha\overline{\phi(0, z_2)}}\phi_2(z_2) (B_1 \phi_1) (z_1) = \alpha \frac{\overline{(B_1\phi_1)(z_1)}(B_1 \phi_1) (z_1)}{1- \alpha\overline{\phi_1(0)\phi_2(z_2)}}.\]
If we expand this as a Fourier series in $L^2,$ we can see that there are no terms with positive $z_2$ power. That means that the projection of this function onto $H^2_2$ must be a constant function. One can compute this directly as:
\[
\begin{aligned} 
& \alpha \left \langle \frac{\overline{(B_1\phi_1)(z_1)}(B_1 \phi_1) (z_1)}{1- \alpha\overline{\phi_1(0)\phi_2(z_2)}}, 1 \right \rangle_{L^2} \\
&=  \alpha \left \langle \overline{(B_1\phi_1)(z_1)}(B_1 \phi_1) (z_1), 1\right \rangle_{L^2} \left \langle \frac{1}{1- \alpha\overline{\phi_1(0)\phi_2(z_2)}},  1 \right \rangle_{L^2} \\ 
& = \| B_1\phi_1 \|^2_{L^2} \frac{\alpha }{1- \alpha\overline{\phi_1(0)\phi_2(0)}} \ne 0.
\end{aligned}
\]
Thus,
\[ P_\phi P_{H^2_2} \overline{\abf{1}_\alpha} F  = P_\phi\left( \| B_1\phi_1 \|^2_{L^2} \frac{\alpha }{1- \alpha\overline{\phi_1(0)\phi_2(0)}} \right). \]
If we are in a situation where the constant functions are not in $K_\phi$ (for example if $\phi_2(0)\ne 0$ as well), then this gives an example where $P_\phi$ is needed to ensure that this operator maps back into $K_\phi.$
 \end{example}

\subsection{A Rational inner function}
Consider the rational inner function
\begin{equation}\label{eq:fave}
\phi(z)=\frac{2z_1z_2-z_1-z_2}{2-z_1-z_2},
\end{equation}
which has a singularity at $(1,1)$ and often serves as an important example in RIF theory (e.g. \cite[Section 3]{Kne15}, \cite[Example 1]{PLMS}, and \cite[Example 5.1]{BickelSolaI}).

Note that $\phi(1,z_2)=\phi(z_1,1)=-1$ so that $\alpha=-1$ is an exceptional value of this RIF. In fact, since the level sets of $\phi$ for other $\alpha$ can be parameterized as
\[\mathcal{C}_{\alpha}=\left\{ \left(\zeta_1, \frac{2\alpha-\alpha\zeta_1+\zeta_1}{2\zeta_1-1+\alpha}\right)\colon \zeta_1 \in \T\right\},\]
this is the sole exceptional value of this function.  A collection of level sets is graphed in Figure \ref{fig:favesupport}.  The Clark measures of $\phi$ are described for $\alpha\neq -1$ by 
\[\int_{\T^2}f(\zeta)d\sigma_{\alpha}=\int_{\T}f\left(\zeta_1, \frac{2\alpha-\alpha\zeta_1+\zeta_1}{2\zeta_1-1+\alpha}\right)\frac{2|\zeta_1-1|^2}{|2\zeta_1-1+\alpha|^2}dm(\zeta_1), \,\,\, f \in C(\T^2),\]
and by Theorem \ref{theorem:poltoratski}, the Clark embedding operator $J_{\alpha}\colon K_{\phi}\to L^2(\sigma_{\alpha})$ is unitary. 
Note that while $\bigcap_{\alpha\in \T}\mathrm{supp}\,\sigma_{\alpha}=\{(1,1)\}$, we have $\sigma_{\alpha}(\{1,1\})=0$ for each $\alpha$, so there is no contradiction to mutual singularity of
the Clark measures. 

\begin{figure}
\includegraphics[width=0.45 \textwidth]{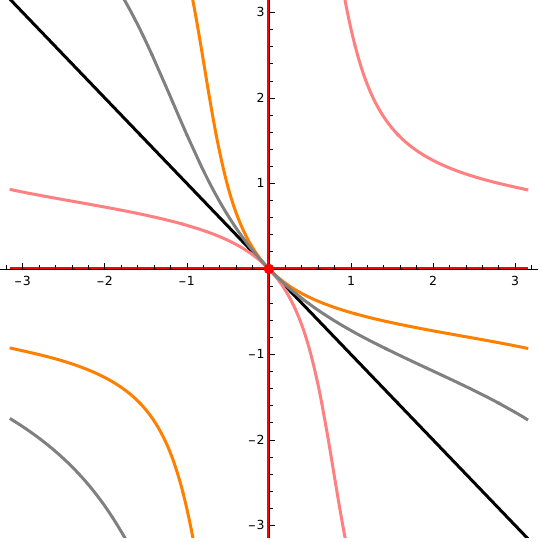}
\caption{  Support sets  in $\mathbb{T}^2\simeq [-\pi, \pi)^2$ for the Clark measures $\sigma_{\alpha}$  for $\phi=\frac{2z_1z_2-z_1-z_2}{2-z_1-z_2}$ and $\alpha=1$ (black), $\alpha=e^{i\frac{\pi}{4}}$ (gray), $\alpha=e^{i\frac{\pi}{2}}$ (orange), $\alpha=e^{-i\frac{\pi}{2}}$ (pink), and the exceptional value $\alpha=-1$ (red). Singular point $(1,1)\simeq (0,0)$ marked in red.}
 \label{fig:favesupport}
\end{figure}

In addition to the level sets of \eqref{eq:fave} containing straight lines in $\mathbb{T}^2$, 
the parameter value $\alpha=-1$ also leads to degeneration in the associated holomorphic functions
\[\abf{1}_{\alpha}(z)=\frac{2\bar{\alpha}}{(1-\bar{\alpha})z_2-2}\cdot\frac{(z_2-1)^2}{2-z_1-z_2}.\]
and $\abf{2}_{\alpha}(z)$.
In the exceptional case, we get
\[\abf{1}_{-1}(z)=\frac{1-z_2}{2-z_1-z_2}\]
which is no longer a bounded rational function in $\D^2$, viz. \cite[p. 1265]{Kne15} (or see \cite[Theorem 1.2]{BKPS} for a more general explanation). See \cite[Example 5.1]{BickelSolaI} for further details including a discussion of what happens in the exceptional case $\alpha=-1$.

The value $\alpha=1$ is generic, and we have
\[\abf{1}_1(z)=-\frac{(z_2-1)^2}{2-z_1-z_2}.\]
The structure of the corresponding Clark measure also simplifies, and we have
\[\int_{\mathbb{T}^2}fd\sigma_1=\int_{\mathbb{T}}f(\zeta_1,\bar{\zeta}_1)\frac{1}{2}|1-\zeta_1|^2dm(\zeta_1).\]
Since $\phi(0,z_2)\neq 0$ and $\phi(z_1,0)\neq 0$, Proposition \ref{prop:ptheta} shows that a projection is needed in the formulas for the unitary operators $U_{1}^1, U_1^2$. By appealing to Agler decompositions, which were discussed in Subsection \ref{aglerdec}, we can nevertheless work out a more concrete representation for the Clark unitaries for $\alpha=1$.

First, we note that since $\deg \phi = (1,1)$ and $\phi$ has a singularity at $(1,1),$ Lemma 4.1 in \cite{BickelGorkin17} implies that $\phi$ has a unique Agler decomposition given by
\begin{equation}
1-\overline{\phi(w)}\phi(z)=
(1-\bar{w}_1z_1)g(z)\overline{g(w)}+(1-\bar{w}_2z_2)h(z)\overline{h(w)},
\label{faveagler}
\end{equation}
where
\[g(z)=\frac{\sqrt{2}(1-z_2)}{2-z_1-z_2} \quad \textrm{and}\quad h(z)=\frac{\sqrt{2}(1-z_1)}{2-z_1-z_2}.\]
Now set
\[ K_1(z,w) = h(z)\overline{h(w)} \ \ \text{ and } \ \ K_2( z, w) = g(z) \overline{ g(w)}.\]
By the uniqueness of the Agler kernels for this $\phi$, it must be the case that $\mathcal{H}(K_1), \mathcal{H}(K_2)$ are closed subspaces of $H^2$ (where the reproducing kernel norm agrees with the $H^2$ norm) 
and thus, $\{h\}$ is an orthonormal basis for $\mathcal{H}(K_1)$ and $\{g\}$ is an orthonormal basis for $\mathcal{H}(K_2)$. Meanwhile, as mentioned in Subsection \ref{aglerdec}, we have the orthogonal decomposition
\[S_1^{\mathrm{max}}=\bigoplus_{k=0}^{\infty}z_1^k\mathcal{H}(K_1) \quad \textrm{and}\quad S_2^{\mathrm{min}}=\bigoplus_{l=0}^{\infty}z_2^l\mathcal{H}(K_2).\]
This implies that $\{z_1^k h\}_{k=0}^{\infty}$ and $\{z_2^lg\}_{l=0}^{\infty}$
are orthonormal bases for $S_1^{\mathrm{max}}$ and $S_2^{\mathrm{min}}$ respectively. We also note that
\[\vartheta_1^1(z)=-\frac{1-z_2}{\sqrt{2}}g(z).\]

We can now determine the action of $U^1_1$ on the orthonormal basis of $S_2^{\mathrm{min}}$. First, we compute, using 
$\|g\|_{H^2}=1$ and $g\perp z_2^kg$, that for $w \in \mathbb{D}^2,$
\begin{align}\notag
\left(P_{H^2_2}M_{\overline{\vartheta_1^1}}gz_2^n\right)(w)&=\left\langle\overline{\vartheta_1^1(z)}g(z)z_2^n ,\frac{1}{1-\bar{w}_2z_2} \right\rangle_{L^2}\\\notag&=\left\langle g(z)z_2^n, \frac{\vartheta_1^1(z)}{1-\bar{w}_2z_2}\right\rangle_{H^2} \\ \notag&=\sum_{k=0}^{\infty}\frac{1}{\sqrt{2}}\langle g(z)z_2^n, (z_2-1)g(z)z_2^k\rangle_{H^2}\, w_2^k\\ \notag&=\frac{1}{\sqrt{2}}\sum_{k=0}^{\infty}(\langle gz_2^{n-k-1},g \rangle_{H^2}-\langle gz_2^{n-k},g, \rangle_{H^2})w_2^k\\ &= \left\{\begin{array}{cc}-\frac{1}{\sqrt{2}}\,, & n=0,\\ \frac{1}{\sqrt{2}}(w_2^{n-1}-w_2^n),& n>0.\end{array} \right.\label{firstproj}
\end{align}
From this we see that $\mathrm{Ran}(P_{H_2^2}M_{\overline{\vartheta_1^1}})\mid_{S_2^{\mathrm{min}}}=H_2^2$. It remains to apply the model space projection $P_{\phi}$ to \eqref{firstproj}. In order to do this, we first use \eqref{faveagler} to get
\[
\begin{aligned}
P_{\phi}(z_2^k)(w)&=
\left\langle z_2^k, k_w^{\phi}\right\rangle_{H^2}\\
&=\left\langle z_2^k, \frac{1-\overline{\phi(w)}\phi(z)}{(1-\bar{w}_1z_1)(1-\bar{w}_2z_2)}\right 
\rangle_{H^2}\\
&=\left\langle z_2^k, \frac{\overline{g(w)}g(z)}{1-\bar{w}_2z_2}+\frac{\overline{h(w)}h(z)}{1-\bar{w}_1z_1} \right\rangle_{H^2}\\ & =\sum_{l=0}^{k}\langle z_2^k, g(z)z_2^l \rangle_{H^2}w_2^lg(w)
+\langle z_2^k,h\rangle_{H^2}h(w);
\end{aligned}
\]
note that the second inner product reduces to one term since powers of $z_1$ and $z_2$ are orthogonal. We now obtain
\begin{align*}
(P_{\phi}P_{H^2_2}M_{\overline{\vartheta_1^2}}g)(w)&=P_{\phi}\left(-\frac{1}{\sqrt{2}}\right)\\
&=-\frac{1}{\sqrt{2}}\left(\langle 1, g\rangle_{H^2} g(w)+\langle 1, h\rangle_{H^2} h(w)\right)\\ & =-\frac{1}{\sqrt{2}}(\overline{g(0)}g(w)+\overline{h(0)}h(w))
\end{align*}
and for $n \ge 1$, 
\begin{align*}
(P_{\phi}P_{H^2_2}M_{\overline{\vartheta_1^2}}gz_2^n)(w)&=\frac{1}{\sqrt{2}}P_{\phi}\left(z_2^{n-1}-z_2^{n}\right)\\
&=\frac{1}{\sqrt{2}}\left(\sum_{l=0}^{n-1}\langle z_ 2^{n-1-l}-z_2^{n-l},g \rangle_{H^2} w_2^l-\overline{g(0)}w_2^n\right)g(w)\\& \quad+\frac{1}{\sqrt{2}}\langle z_2^{n-1}-z_2^n, h \rangle_{H^2}h(w).
\end{align*}
On the other hand, $P_{\phi}P_{H^2_2}M_{\overline{\vartheta_1^2}}|_{S_1^{\mathrm{max}}}=0$ as guaranteed by Theorem \ref{thm:perturbissmall}.

As this relatively simple example illustrates, it does not seem to be easy to find a more concrete representation of Clark unitaries in the case when the underlying inner function has $\phi(0,z_2)\neq 0$ and $\phi(z_1,0)\neq 0$.
\subsection*{Acknowledgements}
The authors thank M. Hartz, N. Jacobsson, M. Jury, E. Krusell, and S. Treil for very helpful discussions and remarks. 

Part of this work was carried out at Mathematisches Forschungsinstitut Oberwolfach and and Centre International de Rencontres Math\'ematiques (CIRM). The authors are grateful to these institutes and their staff for providing such excellent conditions for collaborative mathematical research.

\providecommand{\bysame}{\leavevmode\hbox to3em{\hrulefill}\thinspace}
\providecommand{\MR}{\relax\ifhmode\unskip\space\fi MR }
\providecommand{\MRhref}[2]{%
	\href{http://www.ams.org/mathscinet-getitem?mr=#1}{#2}
}
\providecommand{\href}[2]{#2}

\end{document}